\documentclass[reqno]{amsart}
\usepackage[T1]{fontenc}
\usepackage[utf8]{inputenc}
\usepackage[english]{babel}

\usepackage{blindtext}
\usepackage{amssymb}
\usepackage{hyperref}
\usepackage{xcolor}
\hypersetup{
    colorlinks,
    linkcolor={red!50!black},
    citecolor={blue!50!black},
    urlcolor={blue!80!black}
}
 
\usepackage{amsthm}
\usepackage{amsmath}
\usepackage{mathtools}
\usepackage{dirtytalk}
\newtheorem{theorem}{Theorem}[section]
\newtheorem{prop}[theorem]{Proposition}
\newtheorem{corollary}[theorem]{Corollary}
\newtheorem{lemma}[theorem]{Lemma}

\theoremstyle{remark}
\newtheorem{remark}[theorem]{Remark}

\theoremstyle{definition}
\newtheorem{definition}[theorem]{Definition}
\newcommand{\norm}[1]{\left\lVert#1\right\rVert}
\newcommand\numberthis{\addtocounter{equation}{1}\tag{\theequation}}

\newcommand\blfootnote[1]{%
  \begingroup
  \renewcommand\thefootnote{}\footnote{#1}%
  \addtocounter{footnote}{-1}%
  \endgroup
}
\title[Asymptotic behavior of the spectral radius of cocycles]{Asymptotic behavior of the spectral radius of locally constant strongly irreducible cocycles}
\author{Nicolás Martínez Ramos}

\begin{document}

\maketitle

\begin{abstract}
    We establish some conditions under which $\text{GL}(d,\mathbb{R})$-valued cocycles over a subshift of finite type equipped with an equilibrium state exhibit exponential asymptotics for the spectral radius. Specifically, we show that the exponential growth rate of the spectral radius converges to the top Lyapunov exponent of the cocycle. This result provides a partial answer to a question posed by Aoun and Sert in \cite{aoun}. Our approach relies on large deviation estimates for linear cocycles, which may be of independent interest.  
\end{abstract}

\section{Introduction} 

\blfootnote{March, 2026}
\blfootnote{Mathematics Subject Classification: 37H15.}

Consider a sequence $g_{1},g_{2},\dots$ of independent and identically distributed random $d\times d$ matrices with distribution $\mu$. Furstenberg and Kesten \cite{FurstenbergandKesten} proved that $n^{-1}\log\norm{g_{n}\dots g_{1}}$ converges almost surely to a constant $\lambda_{1}(\mu)\in [-\infty, \infty)$. This constant is called \textit{first Lyapunov exponent} and is finite if the matrices are invertible and their distribution has finite first moment. 

We recall that a probability measure $\mu$ on $\text{GL}(d,\mathbb{R})$ has \textit{finite k-th moment} if 
$$\int \log\max\{\norm{g},\norm{g^{-1}}\}^{k}d\mu(g)<\infty.$$

As a trivial example, if $\mu$ is a Dirac mass at some matrix $g_{0}$, then $\lambda_{1}(\delta_{g_{0}})=\log \rho(g_{0})$, where $\rho(g_{0})$ denotes the \textit{spectral radius} of $g_{0}$, that is, the maximal absolute value of its eigenvalues. 

Cohen \cite{Cohen} inquired about the following variation of the Furstenberg-Kesten theorem, where the norm is replaced by the spectral radius: does the sequence $n^{-1}\log \rho(g_{n}\dots g_{1})$ converge? Moreover, under what conditions does this sequence converge almost surely to the Lyapunov exponent $\lambda_{1}(\mu)$?

A first positive result in this direction was presented by Guivarc'h under several additional hypotheses, namely, a stronger integrability condition (finite exponential moment), as well as strong irreducibility and proximality conditions. See \cite[Théorème~8']{Guivarch_1990}. 

The integrability condition was relaxed to finite second moment by Benoist and Quint \cite[Theorem~14.12]{RandomWalks}. Later, Aoun and Sert \cite{aoun} obtained the same limiting behavior while bypassing the need for strong irreducibility and proximality, needing only finiteness of the second moment.

The theorem of Furstenberg and Kesten was actually stated in the more general setting of non-necessarily independent stationary processes. Let us introduce the equivalent dynamical setting of \textit{linear cocycles}. Let $(\Sigma, \mathcal{B}, \mu, T)$ be an ergodic measure preserving system and let $A:\Sigma\to M_{d\times d}(\mathbb{R})$ be a measurable map. We denote the \textit{cocycle products} as $A^{(n)}(x)=A(T^{n-1}x)\dots A(x)$.

In the context of linear cocycles, the \textit{first Lyapunov exponent} is the almost sure limit $\lambda_{1}(\mu)=\lim_{n\to \infty}n^{-1}\log{a_{1}(A^{(n)}(x))}$, where $a_{1}(\cdot)$ denotes the first singular value. More generally, the \textit{Lyapunov exponents} $\lambda_{1}(\mu)\geq \dots \geq \lambda_{d}(\mu)$ are the constant almost sure limits
$$\lambda_{j}(\mu)=\lim_{n\to\infty}\dfrac{1}{n}\log{a_{j}(A^{(n)}(x))},$$
where $a_{j}(\cdot)$ denotes the $j$-th singular value. 

As an example, note that the previously discussed i.i.d. random products can be modeled using linear cocycles where the measure preserving system is a Bernoulli shift and the fiber map is locally constant. 

Let us resume our discussion on the growth of the spectral radius, now in the setting of linear cocycles.  In this context, Morris \cite[Theorem~1.6]{Morris} showed that for $\mu$-almost every $x\in\Sigma$
\begin{align*}
\limsup_{n\to\infty} \dfrac{1}{n}\log\rho(A^{(n)}(x))=\lambda_{1}(\mu). \numberthis \label{limsup}
\end{align*}
This result was initially proved for $\text{SL}(2,\mathbb{R})$-valued cocycles by Avila and Bochi \cite{Avila2002}. A different proof was given by Oregón-Reyes \cite[Theorem~1.4]{Eduardo}.

We should observe that the well-known theorem of Berger and Wang \cite{BergerWang} on the joint spectral radius is a consequence of Morris' theorem. 

We see that Morris' formula uses a limit superior, while Guivarc'h's is an almost sure limit. The limit (\ref{limsup}) may indeed fail to exist in the general case as observed by Avila and Bochi \cite{Avila2002}. In fact, it already fails to exist for Markov random products, see \cite{aoun}.

This setback prompted Aoun and Sert to ask about the algebraic and measure theoretic conditions on the linear cocycle needed to obtain the expected limit in the case of a Markov random product.

In this article, we give a partial answer to this question since we work in a setting which includes Markov random products. Let us state the main theorems; the appropriate definitions are given later: 

\begin{theorem}\label{TeoremaPrincipal}
    Suppose $(\hat{\Sigma},T)$ is a topologically mixing two-sided subshift of finite type. Suppose that $\hat\mu$ is the equilibrium state of some Hölder potential. Let $A:\hat{\Sigma}\to GL(d,\mathbb{R})$ be a strongly irreducible locally constant linear cocycle such that $\lambda_{1}(\hat{\mu})>\lambda_{2}(\hat{\mu})$. Then, we have the almost sure limit
    $$\lim_{n\to\infty} \dfrac{1}{n}\log{\rho(A^{(n)}(\hat{x}))}=\lambda_{1}(\hat{\mu}).$$
\end{theorem}

 For this purpose, we first define a strong irreducibility assumption which is better suited for the non i.i.d. case. The following definition is inspired by Bougerol \cite{Bougerol1988}:
\begin{definition}
Suppose $(\hat{\Sigma}, T)$ is a subshift of finite type of memory $k$ and $A$ is a locally constant cocycle over $(\hat{\Sigma}, T)$ which is constant on the cylinders of length $k+1$.  Then $A$ is \textit{not strongly irreducible} if there is $1\leq l\leq d-1$ and a finite family of locally constant maps $V_{1},\dots,V_{m}:\hat{\Sigma}\to \text{Grass}(l,\mathbb{R}^{d})$ which is constant on cylinders of length $k$ such that for all $\hat{x}\in \hat{\Sigma}$
    $$A(\hat{x})(V_{1}(\hat{x})\cup\dots\cup V_{m}(\hat{x}))=V_{1}(T\hat{x})\cup\dots\cup V_{m}(T\hat{x}).$$ 
\end{definition}

It should be noted that in the case of locally constant cocycles, the assumption that $A$ is strongly irreducible is weaker than the assumption that the cocyle is pinching and twisting, which is a central hypothesis in the works \cite{bonatti_viana_2004}, \cite{Duarte}, \cite{ParkPiraino} (see Remark \ref{pinchingtwisting}). Additionally, using a perturbation argument, we see that strong irreducibility is a typical property of locally constant cocycles over subshifts of finite type. 

Among our hypotheses, we require a gap between the first two Lyapunov exponents. This assumption is satisfied generically: Indeed, it was proved by Bonatti and Viana \cite{bonatti_viana_2004} that such gap is present in pinching and twisting cocycles. Hence, by Remark \ref{pinchingtwisting}, this hypothesis is typically fulfilled in our setting. 

In any case, we are able to drop this hypothesis if we impose a stronger irreducibility assumption:

\begin{theorem}\label{TeoremaSecundario}
    Suppose $(\hat{\Sigma},T)$ is a topologically mixing two-sided subshift of finite type. Suppose that $\hat\mu$ is the equilibrium state of some Hölder potential. Let $A:\hat{\Sigma}\to GL(d,\mathbb{R})$ be a locally constant linear cocycle such that $\mathsf{\Lambda}^{j}A$ is strongly irreducible for all $1\leq j\leq d$. Then, we have the almost sure limit
    $$\lim_{n\to\infty} \dfrac{1}{n}\log{\rho(A^{(n)}(\hat{x}))}=\lambda_{1}(\hat{\mu}).$$
\end{theorem}

We highlight that Theorems \ref{TeoremaPrincipal} and \ref{TeoremaSecundario} extend the result of Benoist and Quint at least for the case of finite state Markov systems.

Our proof crucially relies on Large Deviation Estimates (LDE) which we obtain by adapting techniques developed by several authors in different settings, see \cite{ParkPiraino}, \cite{Duarte}, \cite{avila2024schrodingeroperatorspotentialsgenerated}. An important component of our approach is a property about zero weight on hyperplanes, which we obtained by a variation of the arguments of Bonatti and Viana \cite{bonatti_viana_2004}.

The organization of the paper is as follows: In Section \ref{Section2} we will establish the notation and give the necessary definitions for the rest of the article. Next, in Section \ref{Section3} we will show the weightlessness of hyperplanes for some measures whose projection is $\hat{\mu}$, which is a necessary tool to develop the Large Deviation Estimates proved in Section \ref{Section4}. In Section \ref{Section5} we will use some geometrical arguments to estimate the distance between subspaces defined by the growth of the cocycle. Finally, in Section \ref{Section6} we will conclude by proving Theorems \ref{TeoremaPrincipal} and \ref{TeoremaSecundario}. 

\section{Preliminaries} \label{Section2}

Let $X$ be a finite set of symbols of cardinality $q$. Recall that all subshifts of finite type are topologically conjugate to a memory $1$ subshift (see \cite[Proposition~2.3.9]{LindandMarcus}). Hence, for the rest of the paper we will assume the topologically mixing two-sided shift $\hat{\Sigma}\subseteq X^{\mathbb{Z}}$ is an edge shift over the symbols of $X$ described by a zero-one adjacency matrix $P$. In particular, $P$ is a primitive $q\times q$ matrix. 

Additionally, we denote by $\Sigma^{-}\subseteq X^{\mathbb{Z}_\leq0}$ and $\Sigma\subseteq X^{\mathbb{Z}_{\geq 0}}$ the one-sided subshifts of finite type to the left and to the right defined by $P$, respectively. 
Moreover, we consider $T:\hat{\Sigma}\to\hat{\Sigma}$ the left shift operator. Note that $T$ is also properly defined for $\Sigma$, while the right shift operator $T^{-1}$ is defined for the shift space in non-positive coordinates $\Sigma^{-}$.

We will use a hat, e.g. $\hat{x}\in \hat{\Sigma}$ to represent two-sided elements. Similarly, we will use a minus sign to represent one-sided elements to the left  e.g. $x^-\in \Sigma^-$. When we do not use a superscript, we will refer to an element which is one-sided to the right, that is $x\in\Sigma$.

\begin{remark}
    Note that both $\Sigma^{-}$ and $\Sigma$ include the zeroth coordinate, hence, it is not true that $\hat{\Sigma}=\Sigma^{-}\times \Sigma$. However, for every element $\xi_{k}\in X$ we can define a homeomorphism $P_{k}$ from $[\xi_{k}]\subseteq \hat{\Sigma}$ to $[\xi_{k}]^{-}\times [\xi_{k}]\subseteq \Sigma^{-}\times \Sigma$ such that $P_{k}(\hat{x})=(P^{-}\hat{x}, P \hat{x})$, where $P^{-},P$ are the projections to $\Sigma^{-},\Sigma$ respectively. Thus, by a small abuse of notation, we will sometimes write $\hat{x}=(x^{-},x)$, where $x^{-}, x$ share the same zeroth coordinate. 
\end{remark}

We consider the metric $\iota: \hat{\Sigma}\times\hat{\Sigma}\to \mathbb{R}$ defined by
$$\iota(\hat{x},\hat{y})=2^{-k},$$
where $k$ is the largest integer such that $x_j=y_j$ for all $|j|\leq k$. We define metrics on $\Sigma^{-}$ and $\Sigma^{+}$ similarly.

For each $x\in\hat{\Sigma}$, we define the \textit{local stable set} $\mathcal{W}^{s}_{loc}(\hat{x})$ and the \textit{local unstable set} $\mathcal{W}^{u}_{loc}(\hat{x})$ as
$$\mathcal{W}^{s}_{loc}(\hat{x})=\{\hat{y}\in\hat{\Sigma}:x_{i}=y_{i} \text{ for } i\geq 0\},$$
$$\mathcal{W}^{u}_{loc}(\hat{x})=\{\hat{y}\in\hat{\Sigma}:x_{i}=y_{i} \text{ for } i\leq0\}.$$
In particular, since $\mathcal{W}^{s}_{loc}(\hat{x})$ depends uniquely on the nonnegative coordinates of the element, for $\hat{x}=(x^{-},x)$ we can define $\mathcal{W}^{s}_{loc}(x)=\mathcal{W}^{s}_{loc}(\hat{x})$. Similarly we can define $\mathcal{W}^{u}_{loc}(x^{-})=\mathcal{W}^{u}_{loc}(\hat{x})$.

Given $x^{-}\in \Sigma^{-}$ and $y\in \Sigma$ with $x^{-}_{0}=y_{0}$, we define the \textit{bracket} between these two elements $[x^{-},y]$ as the unique element in $\mathcal{W}^{u}_{loc}(x^{-})\cap\mathcal{W}^{s}_{loc}(y)$. In general, for any $\hat{x}\in \mathcal{W}^{u}_{loc}(x^{-})$ and $\hat{y}\in \mathcal{W}^{s}_{loc}(y)$, we can denote $[\hat{x},\hat{y}]=[x^{-},\hat{y}]=[\hat{x},y]=[x^{-},y].$ 
\\
\begin{definition}
We say that a $T$-invariant ergodic probability measure $\hat{\mu}$ on $\hat{\Sigma}$ has \textit{continuous local product structure} if there exist a continuous function $\psi:\hat{\Sigma}\to(0,\infty)$ such that on each $0$-cylinder $[\xi]\subset \hat{\Sigma}$
$$d\hat{\mu}=\psi|_{[\xi]}d(\mu^{-}\times\mu),$$
where $P^{-}_{*}\hat{\mu}=\mu^{-}$ and $P_{*}\hat{\mu}=\mu$ are the projections of $\hat{\mu}$ to the one-sided shift spaces $\Sigma^{-}$ and $\Sigma$ respectively.
\end{definition}
Thus for every $x\in \Sigma$ we can define the measure $d\hat{\mu}^{s}_{x}=\psi(\cdot, x)d\mu^{-}$ on the local stable set $\mathcal{W}^{s}_{loc}(x)$. We can similarly define a measure $\hat{\mu}_{x^{-}}^{u}$ on $\mathcal{W}^{u}_{loc}(x^{-})$.

\begin{prop}
    A $T$-invariant probability measure $\hat{\mu}$ on $\hat{\Sigma}$ has continuous local product structure with Hölder density if and only if it is the equilibrium state of a Hölder potential. 
\end{prop}

\begin{proof}\label{equilibriumimpliesproduct}
The continuous local product structure holds in the case of equilibrium states for Hölder potentials on $\hat{\Sigma}$, as discussed in Bonatti and Viana \cite[Section~2.2]{bonatti_viana_2004}. On the other hand, as proved in \cite[Corollary~5]{Haydn}, it is possible to construct a Hölder potential $\varphi$ from the Hölder density of $\hat{\mu}$.
\end{proof}

We will work with a \textit{locally constant linear cocycle} $A$, that is, there exists $m\in \mathbb{N}$ such that $A(\hat{x})$ is constant on any $m$-cylinder $[x_{0},\dots,x_{m-1}]\subseteq \hat{\Sigma}$. Considering the $(m-1)$-block codes of $\hat{\Sigma}$, we see that $A(\hat{x})$ is uniquely determined by the blocks $[x_0\dots x_{m-1}]$ and $ [x_1\dots x_{m}]$. Hence, without loss of generality, we may assume $A$ to be 2-step, that is, $A(\hat{x})$ is constant on each cylinder $[x_{0},x_{1}]\subset \hat{\Sigma}.$ 

Note that since $A$ is 2-step and the cocycle is strongly irreducible, there is not an equivariant finite family of maps $V_{1},\dots, V_{m}: \hat{\Sigma}\to \text{Grass}(l,\mathbb{R}^{d})$ which are constant on each cylinder $[x_{0}]$.

We can construct the projectivized skew product associated to $A$ as
    \begin{align*}
        F_{A}: \hat{\Sigma}\times \mathbb{RP}^{d-1}&\to\hat{\Sigma}\times \mathbb{RP}^{d-1}\\
        F_{A}(\hat{x},\overline{v})&=(T\hat{x}, \overline{A(\hat{x})v}),
    \end{align*}
where the overline denotes the projectivization of the vector underneath.
\begin{definition}
    Consider $\hat{m}$ a probability measure on $\hat{\Sigma}\times \mathbb{RP}^{d-1}$. We call $\hat{m}$ an \textit{invariant s-state} if:
    \begin{itemize}
        \item $\hat\pi_{*}\hat{m}=\hat{\mu}$ where $\hat\pi:\hat{\Sigma}\times\mathbb{RP}^{d-1}\to\hat{\Sigma}$ is the projection to $\hat{\Sigma}.$
        \item $F_{A*}\hat{m}=\hat{m}$.
        \item There exist a disintegration $\{\hat{m}_{\hat{x}}: \hat{x}\in\hat{\Sigma}\}$ such that for any $\hat{x},\hat{y}$ where $\hat{y}\in\mathcal{W}_{loc}^{s}(\hat{x})$, $\hat{m}_{\hat{y}}=\hat{m}_{\hat{x}}.$
    \end{itemize}
\end{definition}

Given $g\in \text{GL}(d, \mathbb{R})$, we take the singular value decomposition $g=L_g \Delta_g R_g$, where $\Delta_g$ is the diagonal matrix of ordered singular values, that is $\Delta_{g}=\text{diag}(a_1(g),\dots, a_{d}(g))$ and $a_1(g)\geq \dots \geq a_{d}(g)$.\\
Let $\{e_1,\dots, e_d\}$ be the canonical basis of $\mathbb{R}^{d}$. We denote the \textit{direction of maximum growth} by
$$u(g)=\overline{L_{g}e_{1}}\in\mathbb{RP}^{d-1}.$$
Moreover, we define the \textit{hyperplane of least growth} by
$$s(g)=\{\overline{v}\in\mathbb{RP}^{d-1}: v\in \text{span}\langle R_{g}^{-1}e_2,\dots, R_{g}^{-1}e_d \rangle\}.$$
This hyperplane is also characterized by the following relation
$$s(g)=\{\overline{w}\in\mathbb{RP}^{d-1}: \langle w, u(g^*)\rangle=0\}.$$
Additionally, we will use the following equality through the proof
$$\dfrac{a_{2}(g)}{a_{1}(g)}=\dfrac{\Vert\mathsf{\Lambda} ^2 g\Vert}{\Vert g\Vert^2}.$$
Finally, we endow $\mathbb{RP}^{d-1}$ with the gap metric. For $\overline{v},\overline{w}\in \mathbb{RP}^{d-1}$ we define 
$$d(\overline{v},\overline{w})=\dfrac{\Vert v\wedge w\Vert}{\Vert v \Vert \Vert w \Vert}=|\sin{\angle(v,w)}|.$$
Then $d$ is a metric, see \cite[Definition~4.1]{Bougerol1985ProductsOR}.
\section{Zero Weight on Hyperplanes} \label{Section3}
\subsection{Continuous local product structure} 
In the first part of this section, we will obtain some consequences for measures with continuous local product structure with Hölder continuous density. The following results are found in \cite{bonatti_viana_2004}. We will add their proofs for the convenience of the reader.

In this subsection, we will assume that $\hat{\mu}$ has continuous local product structure with Hölder density. Moreover, we will denote $P:\hat{\Sigma}\to\Sigma^{-}$ to be the projection of any $\hat{x}\in \hat{\Sigma}$ to its non-positive coordinates, that is, $P(\dots, x_{-2},x_{-1},x_{0},x_{1},x_{2},\dots)=(\dots, x_{-2},x_{-1},x_{0})$.

First, we show that we can find a continuous disintegration of $\hat{\mu}$. 
\begin{lemma} \label{jacobiano1}
    There exists a disintegration $(\hat{\mu}^{u}_{x^{-}})_{x^{-}\in\Sigma^{-}}$ of $\hat{\mu}$ which is continuous with respect to $x^{-}$ in the weak topology. Moreover, the maps
    \begin{align*}
    h_{y^{-}\xleftarrow{}x^{-}}:(\mathcal{W}^{u}_{loc}(x^{-}), \hat{\mu}^{u}_{x^{-}})&\to (\mathcal{W}^{u}_{loc}(y^{-}), \hat{\mu}^{u}_{y^{-}})   \\
    \hat{x}&\mapsto [y^{-},\hat{x}],
    \end{align*}
    are absolutely continuous with Jacobians $J_{y^{-}\xleftarrow{}x^{-}}$ which depend continuously on $x^{-}, y^{-}\in \Sigma^{-}.$
    \end{lemma}
    \begin{proof}
        Due to the continuous local product structure, for every 0-cylinder $[\xi]\subset \hat{\Sigma}$ we have that $d\hat{\mu}|_{[\xi]}=\psi|_{[\xi]}d\mu^{-}\times d\mu$ where $\psi$ is continuous and positive. From $d\mu^{-}=P_{*} d\hat{\mu}$, we obtain that $\int_{\mathcal{W}^{u}_{loc}(x^{-})}\psi(\hat{x})d\mu=1$ for every $x^{-}$. Hence, we get that $d\hat{\mu}_{x^{-}}^{u}=\psi(x^{-},\cdot) d\mu$ is a probability measure and $J_{y^{-}\xleftarrow{}x^{-}}(\hat{x})=\psi(h_{y^{-}\xleftarrow{}x^{-}}(\hat{x}))/\psi(\hat{x})$, as in the statement. 
    \end{proof}

We will proceed to calculate the Jacobian of $\mu^{-}$ with respect to $T^{-1}:\Sigma^{-}\to \Sigma^{-}.$

\begin{lemma} \label{jacobiano2}
    There is a continuous positive Jacobian $J_{\mu^{-}}T^{-1}$ for the the measure $\mu^{-}=P_{*}\hat{\mu}$, that is, $J_{\mu^{-}}T^{-1}(y^{-1})$ is the local ratio of change in measure after applying the dynamics $T^{-1}$ to a small neighborhood of $y^{-}\in \Sigma^{-}$.
\end{lemma}
\begin{proof}
    Denote $y^{-}\in D\subset \Sigma^{-}$, where $D$ is a measurable set contained in the cylinder $[\xi_1\xi_0]$. Then, due to the continuous local product structure and the $T-$invariance of the measure
    \begin{align*}
        \mu^{-}(T^{-1}D)=\hat{\mu}(P^{-1}(T^{-1}D))&=\int_{\{x^{-}\in T^{-1}D,\text{ } w \in [\xi_1]\subset \Sigma\}}\psi(x^{-}, w)d\mu^{-}(x^{-})d\mu(w),\\
        \mu^{-}(D)=\hat{\mu}(P^{-1}D)=\hat{\mu}(T^{-1 }P^{-1}D)&=\int_{\{x^{-}\in T^{-1}D, \text{ } w \in [\xi_1\xi_0]\subset \Sigma\}, }\psi(x^{-}, w)d\mu^{-}(x^{-})d\mu(w). 
        \intertext{Taking $D\to\{y^{-}\}$, we obtain}
        \dfrac{\mu^{-}(T^{-1}D)}{\mu^{-}(D)}&\to\dfrac{\int_{\mathcal{W}^{u}_{loc}(T^{-1}y^{-})}\psi(T^{-1}(y^{-}), w)d\mu(w)}{\int_{\mathcal{W}^{u}_{loc}(T^{-1}y^{-})\cap [\xi_1\xi_0]}\psi(T^{-1}(y^{-}), w)d\mu(w)}\\
        &=\dfrac{1}{\int_{\mathcal{W}^{u}_{loc}(T^{-1}y^{-})\cap [\xi_1\xi_0]}\psi(T^{-1}(y^{-}), w)d\mu(w)}\\
        &=J_{\mu^{-}}T^{-1}(y^{-}).
    \end{align*}
    From the calculation, we see that $J_{\mu^{-}}T^{-1}(y^{-})$ is the expected ratio of measures. Moreover, since $\psi$ is a positive continuous function, we can conclude the same properties for the Jacobian. 
\end{proof}

We define the projection $\Pi:\hat{\Sigma}\times\mathbb{RP}^{d-1}\to \Sigma^{-}\times \mathbb{RP}^{d-1}$ by $\Pi(\hat{x}, \overline{v})=(x^{-},\overline{v})$. Hence, for every s-state, we can define the projection to the one-sided shift by $m=\Pi_{*}\hat{m}$, which we will call an s-state on $\Sigma^{-}\times \mathbb{RP}^{d-1}$. We will show that such one-sided s-states have a continuous disintegration in the weak topology. To obtain that result, we will need the following lemma. 
\begin{lemma} \label{desintegracionm}
    Consider $\hat{m}$ a $T$-invariant measure on $\hat{\Sigma}\times \mathbb{RP}^{d-1}$ which projects to $\hat{\Sigma}$ as $\hat{\mu}$ and define $m=\Pi_{*}\hat{m}$. Consider the partitions $\{\hat\pi^{-1}(\hat{x}): \hat{x}\in \hat{\Sigma}\}$ and $\{\mathcal{W}_{loc}^{u}(x^{-}):x^{-}\in \Sigma^{-})\}$ and the respective disintegration $(\hat{m}_{\hat{x}})_{\hat{x}}$ and $(\hat{\mu}^{u}_{x^{-}})_{x^{-}}$. Then, given the projection $\pi^{-}:\Sigma^{-}\times \mathbb{RP}^{d-1}\to\Sigma^{-}$, we can define a disintegration specified by the partition $\{\pi^{-1}(x^{-}): x^{-}\in \Sigma^{-}\}$ as
    $$m_{x^{-}}=\int \hat{m}_{\hat{x}}d\hat{\mu}^{u}_{x^{-}}(\hat{x})$$
\end{lemma}
\begin{proof}
    Consider $g:\Sigma^{-}\times \mathbb{RP}^{d-1}\to\mathbb{R}$ and $\hat{g}=g\circ \Pi:\hat{\Sigma}\times\mathbb{RP}^{d-1} \to \mathbb{R}$. 
    \begin{align*}
    \int g dm= \int g (\Pi_{*})d\hat{m}=\int \hat{g}d\hat{m}&=\int\left(\int \hat{g}dm_{\hat{x}}\right)d\hat{\mu}(\hat{x})\\
    &=\int\int\left(\int \hat{g}dm_{\hat{x}}\right)d\hat{\mu}^{u}_{x^{-}}(\hat{x})d\mu^-(x^{-}),
    \intertext{since $\hat{g}$ is constant along unstable sets, we obtain that}
    &=\int\left(\int gdm_{\hat{x}}d\hat{\mu}^{u}_{x^{-}}(\hat{x})\right)d\mu^-(x^{-})\\
    &=\int \int g dm_{x^{-}}d\mu^-(x^{-}). 
\end{align*}
The result yields from the calculation.  
\end{proof}

\begin{lemma} \label{weakcontinuous}
    Let $m$ be a s-state on $\Sigma^{-}\times \mathbb{RP}^{d-1}$. Then $m$ admits a continuous disintegration $(m_{x^{-}})_{x^{-}}$ along $\{\pi^{-1}(x^{-}): x^{-}\in \Sigma^{-}\}$. That is, $x^{-}\to \int gdm_{x^{-}}$ is continuous on $\Sigma^{-}$ for all continuous functions $g:\mathbb{RP}^{d-1}\to\mathbb{R}$. 
\end{lemma}
\begin{proof}
    Consider a disintegration $\hat{m}_{\hat{x}}$ of $\hat{m}$ where $m=\Pi_{*}\hat{m}$. Additionally, consider $(d\hat{\mu}^{u}_{x^{-}})_{x^{-}}=(\psi(x^{-}, \cdot)_{x^{-}} d\mu)$ as in Lemma \ref{jacobiano1} and $(m_{x^{-}})_{x^{-}}$ a disintegration of $m$ as in Lemma \ref{desintegracionm}. Consider $x^{-},y^{-}\in [\xi_{0}]\subseteq \Sigma^{-}$ and $\hat{x}\in \mathcal{W}^{u}_{loc}(x^{-})$. Construct $\hat{y}=[y^{-},\hat{x}]$. Then, since $m$ is an s-state, for any continuous function $g:\mathbb{RP}^{d-1}\to\mathbb{R}$
    \begin{align*}
        \int g dm_{y^{-}}&=\int \int gdm_{\hat{y}}d\hat{\mu}^{u}_{y^{-}}(\hat{y})\\
        &=\int \int gdm_{\hat{x}}J_{x^{-}\xleftarrow[]{}y^{-}}(\hat{x})d\hat{\mu}^{u}_{x^{-}}(\hat{x}).
    \end{align*}
    By Lemma \ref{jacobiano1} we know  $J_{x^{-}\xleftarrow[]{}y^{-}}$ is continuous with respect to $x^{-},y^{-}$, hence, for $\iota(x^{-},y^{-})<\delta$ such that $|J_{x^{-}\xleftarrow[]{}y^{-}}(\hat{x})-J_{x^{-}\xleftarrow[]{}x^{-}}(\hat{x})|=|J_{x^{-}\xleftarrow[]{}y^{-}}(\hat{x})-1|<\epsilon$, we obtain 
    \begin{align*}
        \left| \int g dm_{y^{-}}-\int g dm_{x^{-}}\right|&\leq \norm{g}_{\infty}\int |J_{x^{-}\xleftarrow[]{}y^{-}}(\hat{x})-1|dm_{\hat{x}}d\hat{\mu}^{u}_{x^{-}}(\hat{x})\\
        &\leq \norm{g}_{\infty}\epsilon.
    \end{align*}
\end{proof}

\subsection{Weightlessness of hyperplanes}

The following result is based on \cite[Proposition~5.1]{bonatti_viana_2004}, where they obtain a similar conclusion about the weightlessness of planes under pinching and twisting (see Remark \ref{pinchingtwisting}).

\begin{theorem} \label{pesoCero}
    Consider a strongly irreducible 2-step cocycle $A:\hat{\Sigma}\to \text{GL}(d,\mathbb{R})$ whose first Lyapunov exponent is simple. Then, for every invariant s-state  measure $\hat{m}$ on $\hat{\Sigma}\times \mathbb{RP}^{d-1}$ with projection $m=\Pi_{*}\hat{m}$ and continuous disintegration $(m_{x})_{x^{-}\in\Sigma^{-}}$, the measure gives zero weight to hyperplanes, that is
    $$m_{x^{-}}(V)=0,$$
    for every $x^{-}\in\Sigma^{-}$ and every proper projective subspace $V\subseteq \mathbb{RP}^{d-1}$.
\end{theorem}

We will prove this theorem by contradiction. Suppose that there is some $x^{-}\in\Sigma^{-}$
and some minimal $l\in\mathbb{N}$ such that $m_{x^{-}}(V)>0$, where $V$ is a $l$-dimensional proper subspace. In particular, we consider
$$\gamma_{0}=\sup_{V\in \text{Grass}(l,d)}\sup_{x^{-}\in \Sigma^{-}}m_{x^{-}}(V).$$
Since the conditional measures vary continuously with $x^{-}\in\Sigma^{-}$ and $\text{Grass}(d,l)$ is compact, such supremum is attained. 

For completeness of the proof, we will include the following lemmas obtained by Bonnati and Viana \cite{bonatti_viana_2004} which we will need for our setting.
\begin{lemma}\label{lemmazeroweight}
    For every $x^{-}\in\Sigma^{-}$, there exist some $l$-dimensional subspace $V$ such that $m_{x^{-}}(V)=\gamma_{0}$. Also, $m_{x^{-}}(V)=\gamma_{0}$ if and only if  $m_{y^{-}}(A^{(-1)}(y^{-})^{-1}V)=\gamma_{0}$ for every $y^{-}\in T(x^{-})$.
\end{lemma}
\begin{proof}
    Consider $G\subseteq \Sigma^{-}$ the subset of points such that for all $x^{-}\in G$ there exists some $l$-dimensional projective subspace such that $m_{x^{-}}(V)=\gamma_{0}$. From Lemma \ref{weakcontinuous}, we know that the conditional measures $m_{x^{-}}$ are continuous. Thus, since the set of $l$-dimensional projective subspaces is compact, we can infer that $G$ is a closed subset of $\Sigma^{-}$.\\
    Since $\hat{m}$ is $F_{A}$ invariant, then $m$ is invariant for the skew product $F_{A}^{-}(x^{-},v)=(T^{-1}x^{-}, A^{(-1)}(x^{-}))$ defined on $\Sigma^{-}\times \mathbb{RP}^{d-1}$. From the invariance, for $\mu^{-}$-almost every $x^{-}\in\Sigma^{-}$ and any disintegration $(m_{x^{-}})_{x^{-}\in \Sigma^{-}}$ we obtain the equality
    $$\sum_{Tx^{-}=y^{-}}\dfrac{1}{J_{\mu^{-}}T^{-1}(y^{-})}A^{(-1)}(y^{-})_{*}m_{y^{-}}=m_{x^{-}}.$$
    In particular, taking a continuous disintegration, we can ensure that the previous equality happens for all $x^{-}$. 

    We see that $\sum_{Tx^{-}=y^{-}}(J_{\mu^{-}}T^{-1}(y^{-}))^{-1}=1$, therefore we obtain that $m_{x^{-}}(V)=\gamma_{0}$ if and only if $m_{y^{-}}(A^{(-1)}(y^{-})^{-1}V)=\gamma_{0}$ for all $y^{-}\in Tx^{-}$. Thus, $G$ is $T^{-1}$-invariant. Since the subshift is topologically mixing, the backwards orbits of any element are dense in $\Sigma^{-}$. Hence, since $G$ is closed, we conclude that $G=\Sigma^{-}$.
        
\end{proof}

The following lemma gives us a relation between the weight of projective spaces with respect to a disintegration $(\hat{m}_{\hat{x}})_{\hat{x}\in\hat{\Sigma}}$ of $\hat{m}$ compared to a continuous disintegration $(m_{x^{-}})_{x^{-}\in\Sigma^{-}}$ of $m$.

\begin{lemma}
    Given $x^{-}\in\Sigma^{-}$ and $V\in \text{Grass}(l,d)$, we have that $\hat{m}_{\hat{x}}(V)\leq \gamma_{0}$ for $\hat{\mu}^{u}_{x^{-}}$-almost every $\hat{x}\in W^{u}_{loc}(x^{-})$. Hence, $m_{x^{-}}(V)=\gamma_{0}$ if and only if $\hat{m}_{\hat{x}}(V)= \gamma_{0}$ for $\hat{\mu}^{u}_{x^{-}}$-almost every $\hat{x}\in W^{u}_{loc}(x^{-})$.
\end{lemma}
\begin{proof}
    Suppose by contradiction there is $V\in \text{Grass}(l,d)$, $x^{-}\in\Sigma^{-}$, $\gamma_{1}>\gamma_{0}$ and $\Gamma\subseteq \mathcal{W}^{u}_{loc}(x^{-})$ such that $\hat{\mu}^{u}_{x^{-}}(\Gamma)>0$ and $\hat{m}_{\hat{x}}(V)\geq \gamma_{1}$ for every $\hat{x}\in \Gamma$. Consider the partition of $\mathcal{W}^{u}_{loc}(x^{-})$ given by the sets $\{T^{-n}(\mathcal{W}^{u}_{loc}(y^{-})): (y^{-},y)\in T^{n}(\mathcal{W}^{u}_{loc}(x^{-}))\}$. Note that the diameter of each element of the partition vanishes as $n$ grows to infinity. \\
    Hence, by the regularity of the measure $\hat{\mu}_{x^{-}}$, given $\epsilon>0$, there exists $n\in\mathbb{N}$ and $\hat{y}=(y^{-},y)\in T^{n}(\mathcal{W}^{u}_{loc}(x^{-}))$ such that 
    $$\hat{\mu}^{u}_{x^{-}}(\Gamma \cap T^{-n}(\mathcal{W}^{u}_{loc}(y^{-})))\geq (1-\epsilon)\hat{\mu}^{u}_{x^{-}}(T^{-n}(\mathcal{W}^{u}_{loc}(y^{-}))).$$
    In particular, take $\epsilon$ small enough so that $(1-\epsilon)\gamma_{1}\geq \gamma_{0}.$ Since $\hat{m}$ is $F_{A}$-invariant (and therefore also $F_{A}^{-1}$-invariant)
    $$\hat{m}_{\hat{y}}(A^{(-n)}(y{^{-}})^{-1}V)=\hat{m}_{T^{-n}\hat{y}}(V)\geq \gamma_{1},$$
    for $\hat{\mu}$ almost every $\hat{y}\in T^{n}(\Gamma)\cap \mathcal{W}^{u}_{loc}(y^-)$. Thus, using the Jacobian as defined in Lemma \ref{jacobiano2}
    \begin{align*}
        \hat{\mu}^{u}_{y^{-}}(T^{{n}}(\Gamma)\cap \mathcal{W}^{u}_{loc}(y^{-}))&=J_{\mu^{-}}T^{-n}\hat{\mu}^{u}_{x^{-}}(\Gamma\cap T^{-n}(\mathcal{W}^{u}_{loc}(y^{-}))) \\
        &\geq (1-\epsilon)J_{\mu^{-}}T^{-n}\hat{\mu}^{u}_{x^{-}}( T^{-n}(\mathcal{W}^{u}_{loc}(y^{-})))\\
        &=1-\epsilon.
    \end{align*}
    Hence 
    \begin{align*}
        m_{y^{-}}(A^{(-n)}(y^{-})^{-1}V)&=\int_{\mathcal{W}^{u}_{loc}(y^{-})} \hat{m}_{\hat{y}}(A^{(-n)}(y{^{-}})^{-1}V)\hat{\mu}^{u}_{y^{-}}(\hat{y})\\
        &\geq \int_{T^{{n}}(\Gamma)\cap\mathcal{W}^{u}_{loc}(y^{-})} \hat{m}_{\hat{y}}(A^{(-n)}(y{^{-}})^{-1}V)\hat{\mu}^{u}_{y^{-}}(\hat{y})\\
        &\geq \gamma_{1} \int_{T^{{n}}(\Gamma)\cap\mathcal{W}^{u}_{loc}(y^{-})} \hat{\mu}^{u}_{y^{-}}(\hat{y})\\
        &\geq\gamma_{1}(1-\epsilon)\\
        &>\gamma_{0},
    \end{align*}
    contradicting that $\gamma_{0}$ is the supremum of the measures of the $l$-dimensional projective subspaces. This proves the first statement of the assertion. The second affirmation follows from Lemma \ref{desintegracionm}. 
\end{proof}

\begin{proof}[Proof of Theorem \ref{pesoCero}]
    Consider a periodic point of period $m+1$ in the nonnegative coordinates $p^{-}=\prescript{\infty}{}{}(p_m p_{m-1}\dots p_0)\in \Sigma^{-}$ .\\
    Suppose $\{V_{1},\dots, V_{k}\}$ are $l$-dimensional subspaces such that $m_{p^{-}}$ attains the supremum in each of them. Note that $m_{p^{-}}(V_{i}\cap V_{j})=0$ if $i\not=j$ due to the minimality of $l$. Hence
   $$ m_{p^{-}}(V_{1}\cup\dots\cup V_{k})=\sum m_{p^{-}}(V_{i}).$$

Therefore, the set of subspaces where $\gamma_{0}$ is attained is finite. Let us assume that $\{V_{1},\dots, V_{k}\}$ is that set. \\
Construct an element $z^{-}\in [p_{1}p_{0}]\subseteq\Sigma^{-}$ such that there exist $t\in\mathbb{N}$ where $T^{-t}(z)=p^{-}$. Consider the following sets: 
\begin{align*}
    B(p^{-})&=\{ \hat{w}\in \mathcal{W}^{u}_{loc}(p^{-}) \text{ : }\hat{m}_{\hat{w}}(V_{i})\not=\gamma_{0} \text{ for some }1\leq i\leq k \},\\
    B(z^{-})&=\{ \hat{w}\in \mathcal{W}^{u}_{loc}(z^{-}) \text{ : }\hat{m}_{\hat{w}}(V_{i})\not=\gamma_{0} \text{ for some }1\leq i\leq k \}.
\end{align*}
Let us show that if $\hat{w}\in B(p^{-})$, then $\hat{y}=[z^{-},\hat{w}]\in B(z^{-})$. Since $p^{-}, z^{-}$ share the same zeroth coordinate, $\hat{y}$ is well defined. Let $V_i$ be the $l$ dimensional space such that $\hat{m}_{\hat{w}}(V_i)\not=\gamma_{0}$. Notice that $\hat{y}\in W^{s}_{loc}(\hat{w})$, thus, since $\hat{m}$ is an s-state
\begin{align*}
  \hat{m}_{\hat{w}}(V_i)=\hat{m}_{\hat{y}}(V_i)\not= \gamma_{0}.
\end{align*}
Hence, $\hat{y}\in B(z^{-})$. Similarly, if $\hat{y}\in B(z^{-})$, then $\hat{w}=[z^{-},\hat{y}]\in B(p^{-})$. \\
Thus, since $\hat{\mu}_{p^{-}}^{u}$ is equivalent to $\hat{\mu}_{z^{-}}^{u}$, we conclude that $\hat{\mu}_{z^{-}}^{u}(B(z^{-}))=0$. Hence, $m_{z^{-}}(V_i)=\gamma_{0}$ for all $1\leq i\leq k$.\\

Denote $U=\bigcup_{j}V_{j}$. By Lemma \ref{lemmazeroweight}, we know that:
$$U\subseteq (A^{(-t)}(z^{-}))^{-1}U.$$
Hence, by pigeonhole principle, we conclude that $\{V_{1},\dots, V_{k}\}$ are the only projective subspaces of dimension $l$ where the measure is $\gamma_{0}$ for $m_{z^{-}}$.\\
Now, we define the following map on the set of symbols $X$. For $\xi\in X$, we consider any element $x^-$ such that 
$$x^{-}=\prescript{\infty}{}{}(p_m p_{m-1}\dots p_0)\dots\xi\underbrace{\beta_{s-1}\dots \beta_1p_0}_{\text{s-symbols}},$$
and we define
$$L_{x^{-}}(\xi)=A^{(-s)}(x^-)U,$$
were $s\in\mathbb{N}$ is the minimum instance such that $(T^{-s}x^{-})_{0}=\xi$. Let us show that the previous description is well defined and does not depend on $x^{-}$. Consider two elements
\begin{align*}
    x_{1}^{-}&=\prescript{\infty}{}{}(p_m p_{m-1}\dots p_0)r_{s_2}\dots r_1\xi \beta_{s_1-1}\dots \beta_{1}p_{0},\\
    x_{2}^{-}&=\prescript{\infty}{}{}(p_m p_{m-1}\dots p_0)v_{s_4}\dots v_1\xi \delta_{s_3-1}\dots \delta_{1}p_{0},
    \end{align*}
where $\beta_{1},\dots,\beta_{s_1-1},\delta_{1},\dots,\delta_{s_3-1}$ are all symbols different to $\xi\in X$. We have to prove that $L_{x_1^{-}}(\xi)=L_{x_2^{-}}(\xi)$. Using the fact that our subshift is of memory one, we can construct:
$$x_{3}^{-}=\prescript{\infty}{}{}(p_m p_{m-1}\dots p_0)r_{s_2}\dots r_1\xi \beta_{s_1-1}\dots \beta_{1}p_{0}.$$
It is clear that $L_{x_{1}^{-}}(\xi)=L_{x_{3}^{-}}(\xi)$. Moreover, since $x_{3}^{-}$ is a homoclinic point, by the previous discussion
\begin{align*}
    U&= (A^{(-s_{1}-s_{4}-1)}(x^{-}_{3}))^{-1}U\\
    &=(A^{(-s_4-1)}(T^{-s_{1}}x^{-}_{3})A^{(-s_{1})}(x^{-}_{3}))^{-1}U.\\
    \intertext{Hence}
     A^{(-s_{1})}(x^{-}_{3})U&= (A^{(-s_4-1)}(T^{-s_{1}}x^{-}_{3}))^{-1}U\\
     L_{x_{1}^{-}}(\xi)&= (A^{(-s_4-1)}(T^{-s_{1}}x^{-}_{3}))^{-1}U.
\end{align*}
Similarly, since $x_{2}^{-}$ is also a homoclinic point
\begin{align*}
     L_{x_{2}^{-}}(\xi)&= (A^{(-s_4-1)}(T^{-s_{3}}x^{-}_{2}))^{-1}U.
\end{align*}
Since $T^{-s_3}(x^{-}_{2})=T^{-s_1}(x^{-}_{3})$, we conclude that $L_{x^{-}_{1}}(\xi)=L_{x_{2}^{-}}(\xi)$. Hence, $L$ is well defined and it only depends on symbols.\\
We conclude that $L(\xi)$ is equivariant. Indeed, consider an allowed 2-word $\xi_1\xi_2$. Using the topologically mixing property of the system, we can construct an element
$$x^{-}=\prescript{\infty}{}{}(p_m p_{m-1}\dots p_0)\dots \xi_{2}\xi_{1}\underbrace{\beta_{s-1}\dots \beta_1p_0}_{\text{s-symbols}}$$
Hence
$$L(\xi_{2})=A^{(-s-1)}(x^{-})\{V_{1},\dots, V_{k}\}=A^{(-1)}(T^{-s}x^{-})L(\xi_{1}).$$ 
We conclude $A$ is not strongly irreducible, which is a contradiction. 
\end{proof}

\begin{remark}\label{pinchingtwisting}
Assume that $A$ is a pinching and twisting  2-step cocycle over a subshift of finite type of memory 1 with pinching periodic point $\hat{p}\in\hat{\Sigma}$ of period $l\in\mathbb{N}$ and homoclinic point $\hat{z}\in \mathcal{W}_{loc}^{u}(\hat{p})$. By conjugating if necessary, we may assume that $A^{(l)}(\hat{p})$ is a diagonal matrix. Similarly, we note that the stable and unstable holonomies are trivial since the cocycle is 2-step. 
    
Suppose by contradiction that there is a finite family of invariant locally constant maps $V_1(\hat{x}),\dots, V_{m}(\hat{x})\in \text{Grass}(l,\mathbb{R}^{d})$ which only depend on the zeroth coordinate. Since $A^{(l)}(\hat{p})$ is a diagonal matrix, then $V_{i}(\hat{p})=V_{i}(p_0)$ are proper subspaces orthogonal to some coordinate axis for $1\leq i\leq m$. On the other hand, since there is $k\in\mathbb{N}$ such that $T^{k}\hat{z}\in\mathcal{W}_{loc}^{s}(\hat{p})$ and the twisting operator is $\psi_{\hat{p},\hat{z}}=A^{k}(\hat{z})$  due to the holonomies being trivial, we obtain that
    \begin{align*}
      A^{k}(\hat{z})(V_{1}(\hat{z})\cup\dots\cup V_{m}(\hat{z}))&=A^{k}(\hat{z})(V_{1}(p_0)\cup\dots\cup V_{m}(p_0))\\
      &=V_{1}(T^{k}\hat{z})\cup\dots\cup V_{m}(T^{k}\hat{z})\\
      &=V_{1}(p_0)\cup\dots\cup V_{m}(p_0).
    \end{align*}
Since an eigenbasis basis of $A^{(l)}(\hat{p})$ is given by the standard basis of $\mathbb{R}^{d}$, the twisting assumption yields a contradiction. 
\end{remark}

\section{Large Deviation Estimates}  \label{Section4}
In this section we will proceed to find some Large Deviation Estimates (LDE) for cocycles. The following result is based on a similar LDE presented in \cite[Theorem~3.6]{avila2024schrodingeroperatorspotentialsgenerated}.

Recall that the space of Hölder continuous functions of exponent $\alpha \in (0,1)$, denoted $\mathcal{H}_{\alpha}(\Sigma)$ is a Banach space of norm:
$$\norm{g}_{\alpha}=\norm{g}_{\infty}+\sup\left\{\dfrac{|g(x)-g(y)|}{\iota(x,y)^{\alpha}}: x,y\in \Sigma, x\not=y\right\}.$$
Additionally, we will denote the $(i+1)$-cylinder around $\omega\in \Sigma$ as
    $$D_{i}(\omega)=[w_0,\dots,w_{i}]\subseteq \Sigma.$$
We will need the following result about exponential mixing.

\begin{lemma}\label{exponentialmixing}
Suppose $\mu$ is the equilibrium state of a H\"older potential on $\Sigma$. Fix $\omega\in \Sigma$ and $\alpha\in (0,1)$. Then there are constants $C>0$, $0<\gamma<1$ such that for all $g\in \mathcal{H}_{\alpha}(\Sigma)$, $f(x)=\dfrac{\chi_{D_{i}(\omega)}}{\mu(D_{i}(\omega))}(x)$ and $i\in \mathbb{N}$, we have
$$\left|\dfrac{1}{\mu(D_{i}(\omega))}\int_{D_{i}(\omega)} g\circ T^{i}d\mu-\int g d\mu\right|\leq C\gamma^{i}\left\Vert g-\int gd\mu\right\Vert_{\alpha}.$$
\end{lemma}

\begin{proof}
    Recall that $\mu=\pi_{*}\hat{\mu}$, where $\pi$ is the projection of $\hat{\Sigma}$ onto $\Sigma$. Hence
    \begin{align*}
        \left|\int f(g\circ T^{i})d\mu-\int g d\mu\right|&=\left|\int (f\circ \pi)(g\circ T^{i}\circ \pi)d\hat{\mu}-\int g\circ \pi d\hat{\mu}\right|,
        \intertext{since the projection and the shift \say{commute}, taking into account the change from one-sided shift to two-sided shift, we obtain}
        \left|\int f(g\circ T^{i})d\mu-\int g d\mu\right|&=\left|\int (f\circ \pi)(g\circ \pi\circ T^{i})d\hat{\mu}-\int g\circ \pi d\hat{\mu}\right|\\
        \intertext{using the $T^{-1}$ invariance of the two-sided shift and $\norm{f}_{L^1}=1$}
        \left|\int f(g\circ T^{i})d\mu-\int g d\mu\right|&=\left|\int (f\circ \pi\circ T^{-i})(g\circ \pi)d\hat{\mu}-\int g\circ \pi d\hat{\mu}\right|\\
        &=\left|\int (f\circ \pi\circ T^{-i})(g\circ \pi)d\hat{\mu}-\int f\circ \pi d\hat{\mu} \int g\circ \pi d\hat{\mu}\right|.
    \end{align*}    
    Hence, using the exponential decay of correlations of equilibrium states as stated in \cite[Theorem~5.4.9]{urbanski}
    \begin{align*}
    \left|\int f(g\circ T^{i})d\mu-\int g d\mu\right|&\leq C\gamma^{i} \left\Vert f-\int fd\mu\right\Vert_{L^1}\left \Vert g-\int gd\mu\right\Vert_{\alpha}\\
    &\leq 2C\gamma^{i}\left \Vert g-\int gd\mu \right\Vert_{\alpha}    
    \end{align*}
\end{proof}

The following result shows that outside an exponentially small set, the exponential growth of the norm of the cocycle cannot be larger than a slight increment of the first Lyapunov exponent. Note that the unique hypothesis we are imposing on $A$ is that it is a locally constant cocycle. Indeed, in Section \ref{Section5} we will use this result on the cocycle $\mathsf{\Lambda}^2 A$, which is not necessarily strongly irreducible and may not satisfy $\lambda_{1}(\mathsf{\Lambda}^2 A,\hat{\mu})>\lambda_2(\mathsf{\Lambda}^2 A,\hat{\mu})$.

\begin{lemma} \label{PrimerLDE}
    Let $\mu$ be the equilibrium state of a Hölder potential and consider a two-step cocycle $A$. Then, for every $\epsilon>0$ there exist $M>0, \beta>0$ such that for all $n\in\mathbb{N}$
    $$\mu\left\{x\in\Sigma: \frac{1}{n}\log\norm{A^{(n)}(x)}-\lambda_{1}(A,\hat{\mu})>\epsilon \right\}<Me^{-\beta n}.$$
\end{lemma}

It should be noted that Lemma \ref{PrimerLDE} can be derived from the more general result \cite[Proposition~3.1]{klein} combined with the LDEs for observables obtained in \cite[Theorem~5.1]{Duarte}. Nonetheless, we prefer to give a self-contained proof based on martingale arguments. 

\begin{proof}
    By rescaling the cocycle, we may assume that $\lambda_1(A,\hat{\mu})=0$. Given $\epsilon>0$, fix $N\in \mathbb{N}$ such that
    $$\left|\int\frac{1}{N}\log\norm{A^{(N)}}d\mu\right|<\frac{\epsilon}{4}.$$
    Using Lemma \ref{exponentialmixing} we can find $C>0$ and $0<\gamma<1$ such that for all $i\in\mathbb{N}$ and all $\omega\in\Sigma$
    \begin{align*}
    &\left|\dfrac{1}{\mu(D_{i}(\omega))}\int_{D_{i}(\omega)}\log\norm{A^{(N)}\circ T^{i}}d\mu-\int\log\norm{A^{(N)}}d\mu\right|\\
    &\leq C\gamma^{i}\left\Vert\log\norm{A^{(N)}}-\int \log\norm{A^{(N)}}d\mu\right\Vert_{\alpha},
    \end{align*}
    where $\alpha\in (0,1)$ is such that $A^{(N)}\in \mathcal{H}_{\alpha}(\Sigma)$. We know that such $\alpha$ exists since $A^{(N)}$ only depends on finitely many coordinates. 
    
    Thus, taking $k$ such that $C\gamma^{k}\norm{\log\norm{A^{(N)}(x)}-\int \log\norm{A^{(N)}(x)}d\mu}_{\alpha}<\frac{\epsilon}{4}$, we obtain for all $i\geq k$
    \begin{align*}
        &\left|\dfrac{1}{N \mu(D_{i}(\omega))}\int_{D_{i}(\omega)}\log\norm{A^{(N)}(T^{i}x)}d\mu\right|\\
        &\leq \dfrac{1}{N}\left|\int\dfrac{\chi_{D_{i}(\omega)}}{\mu(D_{i})}\log\norm{A^{(N)}(T^{i}x)}d\mu-\int\log\norm{A^{(N)}(x)}d\mu\right|+\left|\int\frac{1}{N}\log\norm{A^{(N)}(x)}d\mu\right|\\ \numberthis \label{desigualdad1}
        &<\dfrac{\epsilon}{2}. 
    \end{align*}
    Denote $\mathcal{B}_{i}$ to be the $\sigma$-algebra generated by the set of cylinders of length $i+1$.
    Consider the random variables $Y_{i}:\Sigma\to\mathbb{R}$ defined by $Y_{i}(x):=\log\norm{A^{(i)}(x)}$, which are $\mathcal{B}_{i}$-measurable. Note that $\mathbb{E}(Y_{i}|\mathcal{B}_{i})=Y_{i}$ and we can rewrite $Y_{i}$ as 
    $$Y_i(x)=\dfrac{1}{\mu(D_{i}(x))}\int_{D_{i}(x)}\log\norm{A^{(i)}(z)}d\mu(z).$$
    The conditional expectation of $Y_{N+i}$ with respect to $\mathcal{B}_{i}$ is
    \begin{align*}
        \mathbb{E}(Y_{N+i}|\mathcal{B}_{i})(\omega)=\dfrac{1}{\mu(D_{i}(\omega))}\int_{D_{i}(\omega)}\log\norm{A^{(N+i)}(z)}d\mu(z).
    \end{align*}
    Thus, using \eqref{desigualdad1}, we obtain
    \begin{align*}
        \dfrac{1}{N}\mathbb{E}(Y_{N+i}-Y_{i}|\mathcal{B}_i)&=\dfrac{1}{N \mu(D_{i}(\omega))}\int_{D_{i}(\omega)}\log\dfrac{\norm{A^{(N+i)}(z)}}{\norm{A^{(i)}(z)}}d\mu(z)\\
        &\leq\dfrac{1}{N \mu(D_{i}(\omega))}\int_{D_{i}(\omega)}\log \norm{A^{(N)}(T^{i}z)}d\mu(z)\\
        &\leq \dfrac{\epsilon}{2}. \numberthis \label{desigualdad2Version1}
    \end{align*}
    Consider $t\in \mathbb{N}$ such that $tN\geq k$. For $n>t$ define the sequence of random variables 
    $$X_{n}=Y_{nN}-\sum_{l=t+1}^{n}\mathbb{E}(Y_{lN}-Y_{(l-1)N}|\mathcal{B}_{(l-1)N}).$$
    Note that
    $$X_{n+1}-X_{n}= Y_{(n+1)N}-\mathbb{E}(Y_{(n+1)N}|\mathcal{B}_{nN}).$$
    In particular $\left\{X_{n}\right\}$ is a martingale difference adapted to $\left\{\mathcal{B}_{nN}\right\}$ that is
    \begin{align*}
        \mathbb{E}(X_{n+1}-X_{n}|\mathcal{B}_{nN})&=\mathbb{E}\left(Y_{(n+1)N}-\mathbb{E}(Y_{(n+1)N}|\mathcal{B}_{nN}) \Big{\rvert}\mathcal{B}_{nN}\right)\\
        &=0.
    \end{align*}
    We want to use the Azuma bound for martingale differences as it appears in Alon and Spencer  \cite[Theorem~7.2.1]{Alon2000}, therefore we need to check that the difference $X_{n+1}-X_{n}$ is uniformly bounded for all $n>t$.\\
    To find such bound, note that we can express $D_{nN}(\omega)=\sqcup_{j}D_{(n+1)N}(x_{j})$ for appropriate $x_{j}\in \Sigma$ which depend on $j, n$ and $\omega$. Hence
    \begin{align*}
        &|X_{n+1}(\omega)-X_{n}(\omega)|= \left|Y_{(n+1)N}(\omega)-\mathbb{E}(Y_{(n+1)N}|\mathcal{B}_{nN})(\omega)\right|\\
        &=\left|\log\norm{A^{((n+1)N)}(\omega)}-\dfrac{1}{\mu(D_{nN}(\omega))}\int_{D_{nN}(\omega)}\log\norm{A^{((n+1)N)}(z)}d\mu(z)\right|\\
        &=\left|\log\norm{A^{((n+1)N)}(\omega)}-\sum_{j}\dfrac{\mu(D_{(n+1)N}(x_{j}))}{\mu(D_{nN}(\omega))}\log\norm{A^{((n+1)N)}(x_{j})}\right|\\ 
        &\leq \sum_{j}\dfrac{\mu(D_{(n+1)N}(x_{j}))}{\mu(D_{nN}(\omega))}\left|\log\norm{A^{((n+1)N)}(\omega)}-\log\norm{A^{((n+1)N)}(x_{j})}\right|. \\
    \end{align*}
    Let $m:\text{GL}(d,\mathbb{R})\to\mathbb{R}$ be the co-norm, that is, the least singular value. Note that for matrices $P, Q_{1}, Q_{2}\in \text{GL}(d,\mathbb{R})$
    $$\left|\log\norm{Q_{1}P}-\log\norm{Q_{2}P} \right|\leq \log\max\left\{\dfrac{\norm{Q_{1}}}{m(Q_{2})},\dfrac{\norm{Q_{2}}}{m(Q_{1})}\right\}.$$
    Hence, if we denote $\gamma_{j}=\dfrac{\mu(D_{(n+1)N}(x_{j}))}{\mu(D_{nN}(\omega))}$, we obtain
    \begin{align*}
        |X_{n+1}-X_{n}|&\leq \sum_{j}\gamma_{j}\left|\log\dfrac{\norm{A^{((n+1)N)}(\omega)}}{\norm{A^{((n+1)N)}(x_{j})}}\right|\\
        &\leq\sum_{j}\gamma_{j}\max_{z_{1},z_{2}\in\Sigma}\log\dfrac{\norm{A^{(N)}(z_{1})}}{m(A^{(N)}(z_{2}))}\\
        &\leq a,
    \end{align*}
    where $a>0$ is a constant which depends on $N$. As a consequence, using Azuma's bound, we obtain that
    \begin{align}
        \mu\left\{\frac{1}{n}X_{n}>\epsilon\right\}<e^{-\frac{n\epsilon^{2}}{2a^{2}}}. \label{AzumaVersion1}
    \end{align}
    Our objective now is to estimate the measure of the subset of $\Sigma$ of elements such that the norm of the cocycle grows exponentially faster than the first Lyapunov exponent. Thence, consider $\omega\in \Sigma$ such that
    $$\dfrac{1}{nN}\log\norm{A^{(nN)}(\omega)}>\epsilon.$$
    Note that this is the same as
    $$\dfrac{1}{nN}Y_{nN}(\omega)>\epsilon.$$
    Hence, using \eqref{desigualdad2Version1}
    \begin{align*}
        \dfrac{1}{nN}X_{n}(\omega)&=\dfrac{1}{nN}\left(Y_{nN}(\omega)-\sum_{l=1}^{n}\mathbb{E}(Y_{lN}-Y_{(l-1)N}|\mathcal{B}_{(l-1)N})(\omega)\right)\\
        &\geq \dfrac{1}{nN}Y_{nN}(\omega)-\dfrac{1}{n}\sum_{l=t+1}^{n}\left(\dfrac{1}{N}\mathbb{E}(Y_{lN}-Y_{(l-1)N}|\mathcal{B}_{(l-1)N})(\omega)\right)\\
        &>\frac{\epsilon}{2}. 
    \end{align*}
    Thus, we can conclude from \eqref{AzumaVersion1} that for $n>t$
    \begin{align*}
        \mu\left\{\omega:\dfrac{1}{nN}\log\norm{A^{(nM)}(\omega)}>\epsilon\right\}&\leq \mu\left\{\omega:\dfrac{1}{n}X_{n}>\dfrac{N\epsilon}{2}\right\}\\
        &\leq Me^{-\frac{N\epsilon^2}{8a^{2}}(nN)}.
    \end{align*}
    Consider $s\geq t$ such that
    $$(s-1)N\epsilon\geq \max_{\substack{x\in\Sigma\\\vspace{0.15cm} 1\leq l\leq N-1}}\log\norm{\left(A^{(l)}(x)\right)^{-1}}.$$
    Then, we find the following bounds for $n\geq s$ when $1\leq l\leq N-1$
    \begin{align*}
        &\mu\left\{\log\norm{A^{(nN+l)}(x)}>2(nN+l)\epsilon\right\} \\
        &=\mu\left\{\log\norm{(A^{(N-l)}(T^{nN+l}x))^{-1}A^{((n+1)N)}(x)}>2(nN+l)\epsilon\right\} \\
        &\leq\mu\left\{\log\norm{A^{((n+1)N)}(x)}>2(nN+l)\epsilon-\log\norm{(A^{(N-l)}(T^{nN+l}x))^{-1}}\right\}\\
        &\leq \mu\left\{\log\norm{A^{((n+1)N)}(x)}>(n+1)N\epsilon\right\}\\
        &\leq Me^{-\frac{N\epsilon^2}{8a^{2}}((n+1)N)}\\
        &\leq Me^{-\frac{N\epsilon^2}{8a^{2}}(nN+l)}.
    \end{align*}
    We conclude by sufficiently enlarging the constant $M=M(N)$ so that the statement is true for $0\leq m\leq sN-1$.
\end{proof}

Although the previous result is general, we require a stronger LDE. Before we proceed with the LDE, we have to set up some notation. 

Consider the adjoint and inverse cocycles respectively as
$$A_{*}(\hat{x})=A(T^{-1}\hat{x})^{*} \text{ and } A^{-1}(\hat{x})=A(T^{-1}\hat{x})^{-1},$$
over $(\hat{\Sigma}, T^{-1})$. Note that since the cocycle only depends on the zeroth and first coordinate, it is possible to define the original cocycle on the positive one-sided (left) shift space $\Sigma$, while the inverse and adjoint cocycles are defined over the negative one-sided (right) shift $\Sigma^{-}$.

We also define the \text{inverse adjoint cocycle} on $(\hat{\Sigma}, T)$ as
$$A_{*}^{-1}(\hat{x})=(A(\hat{x})^{-1})^{*}.$$
\begin{remark} \label{primerRemark}
Notice that if $A$ is not strongly irreducible, neither are the cocycles $A^{-1}$ and $A_{*}$. Indeed, if we consider an invariant family of maps $V_{1}(\hat{x}),\dots, V_{m}(\hat{x})$ as in the definition, it is readily seen that the same family of maps is invariant for $A^{-1}$ and the family $V_{1}^{\perp}(\hat{x}),\dots, V_{m}^{\perp}(\hat{x})$ is invariant for $A_{*}$. Thus, we see that $A^{-1}_{*}$ is strongly irreducible if and only if $A$ is strongly irreducible.
\end{remark}
Consider a singular value decomposition
\begin{align*}
    A^{(n)}(\hat{x})&=L_{n}(\hat{x})\Delta_{n}(\hat{x})R_{n}(\hat{x}),
    \intertext{where $\Delta_{n}(\hat{x})=\text{diag}(a_{1}(\hat{x})\geq\dots\geq a_{d}(\hat{x}))$ and $L_{n},R_{n}\in O(d)$. Then}
    A^{(-n)}_{*}(\hat{x})&=L_{n}(\hat{x})\Delta^{-1}_{n}(\hat{x})R_{n}(\hat{x}).
\end{align*}
Let $R_{\infty}(\hat{x})$ be a limit point of the sequence $\{R_{n}(\hat{x})\}_{n\in \mathbb{N}}$. Then, we define $\xi(\hat{x})=R_{\infty}^{*}(\hat{x})e_{1}$ where $R_{\infty}^{*}(\hat{x})$ is the adjoint of $R_{\infty}(\hat{x})$. When the top Lyapunov exponent $\lambda_{1}$ is simple, we can infer that $\xi(\hat{x})$ is the \textit{bottom one-dimensional Oseledets space} at $\hat{x}$ for $A^{-1}_{*}$:
$$\xi(\hat{x})=\left\{\overline{v}\in\mathbb{RP}^{d-1}: \lim\dfrac{1}{n}\log\norm{A^{(-n)}_{*}(\hat{x})v}=-\lambda_{1}\right\}.$$

We define a measure on $\hat{\Sigma}\times \mathbb{PR}^{d-1}$ as
$$\hat{m}=\int\delta_{\xi(\hat{x})}d\hat{\mu}(\hat{x}).$$
Note that $\hat{m}$ is  invariant under $F_{A_{*}^{-}}$. Moreover, note that we can disintegrate the measure $\hat{m}$ as $\{\hat{m}_{\hat{x}}=\delta_{\xi(\hat{x})}\}$. Then, for $\hat{x},\hat{y}$ such that $\hat{y}\in W^{s}_{loc}(\hat{x})$, since $A_{*}^{-1}$ uniquely depends on the future and it is locally constant, the stable holonomy connecting $\hat{x}$ and $\hat{y}$ is trivial and
\begin{align*}
    \xi(\hat{x})&=\left\{\overline{v}\in\mathbb{RP}^{d-1}: \lim\dfrac{1}{n}\log\norm{A^{(-n)}_{*}(\hat{x})v}=-\lambda_{1}\right\}\\
    &=\left\{\overline{v}\in\mathbb{RP}^{d-1}: \lim\dfrac{1}{n}\log\norm{A^{(-n)}_{*}(\hat{y})v}=-\lambda_{1}\right\}\\
    &=\xi(\hat{y}).
\end{align*}
Hence, we have that
$$\hat{m}_{\hat{x}}=\hat{m}_{\hat{y}}.$$
We conclude that $\hat{m}$ is an invariant s-state. Before we continue with the next lemma, note that if we define $\Pi_{*}\hat{m}=m$ as before, taking a disintegration, for $x^{-}\in\Sigma^{-}$ we obtain
$$m_{x^{-}}=\int \delta_{\xi(\hat{x})}d\hat{\mu}_{x^{-}}^{u}(\hat{x}).$$

The following lemma is a core component to obtain a stronger LDE. In particular, this is the step of the proof where strong irreducibility is required since we need to apply the weightlessness of hyperplanes we obtained in Theorem \ref{pesoCero}. We note that a similar result was obtained in \cite[Proposition~6.3]{Duarte} in the context of twisting and pinching cocycles which are not necessarily locally constant. Correspondingly, their proof also requires zero weight on hyperplanes, which was proved in \cite[Proposition~5.1]{bonatti_viana_2004}.

\begin{lemma} \label{LemmaIntegral}
    Let $\hat{\mu}$ be the equilibrium state of a Hölder potential. Consider a strongly irreducible 2-step cocycle $A$ such that $\lambda_1(\hat{\mu})>\lambda_2(\hat{\mu})$. Then for any unit vector $v\in\mathbb{R}^{d}$, 
    $$\lim_{n\to\infty}\dfrac{1}{n}\int\log\norm{A^{(n)}(\hat{x})v}d\hat{\mu}(\hat{x})=\lambda_1(A,\hat{\mu}).$$
    Moreover, the previous limit is uniform on unit vectors.
\end{lemma}
\begin{proof}
    Consider $\Omega\subseteq \hat{\Sigma}$ the full measure set of points $\hat{x}$ such that $\xi(\hat{x})$ is well defined. Then, if $\{v_{n}\}_{n\in\mathbb{N}}$ is a sequence of unit vectors such that $v_{n}\to v$ as $n\to\infty$ and we take the singular value decomposition as before, we obtain for $\hat{x}\in \Omega$
    \begin{align*}
        \lim_{n\to\infty}\dfrac{\norm{A^{(n)}(\hat{x})v_{n}}}{\norm{A^{(n)}(\hat{x})}}=&\lim_{n\to\infty}\dfrac{\norm{L_{n}(\hat{x})\Delta_{n}(\hat{x})R_{n}(\hat{x})v_{n}}}{\norm{L_{n}(\hat{x})\Delta_{n}(\hat{x})R_{n}(\hat{x})}}\\ \numberthis \label{limiteOseledets}
        =&\lim_{n\to\infty}\dfrac{\norm{\Delta_{n}(\hat{x})R_{n}(\hat{x})v_{n}}}{\norm{\Delta_{n}(\hat{x})}}\\
        =&\langle R_{\infty}(\hat{x})v,e_{1}\rangle\\
        =&\langle v, R_{\infty}^{*}(\hat{x})e_{1}\rangle\\
        =&\langle v, \xi(\hat{x})\rangle.
    \end{align*}
Consider the set
$$N(x^-)=\left\{\hat{x}\in \mathcal{W}^{u}_{loc}(x^-): \lim_{n\to\infty}\dfrac{\norm{A^{(n)}(\hat{x})v_{n}}}{\norm{A^{(n)}(\hat{x})}}\not= \langle v, \xi(\hat{x})\rangle \text{ or the limit does not exist}\right\}.$$
Note that if there exist $x^{-}\in\Sigma^{-}$ such that $\hat{\mu}_{x^{-}}^{u}(N(x^{-}))>0$, then for all $y^{-}\in [x_{0}]\subset \Sigma^{-}$ we have $\hat{\mu}_{y^{-}}^{u}(N(y^{-}))>0$ since $\hat{\mu}_{x^{-}}^{u}\approx \hat{\mu}_{y^{-}}^{u}$. Then the equality (\ref{limiteOseledets}) does not happen in a set of positive measure. But we just checked the equality on a full measure set. Thus, such $x^{-}$ does not exist.

Given any $x^{-}\in \Sigma^{-}$ and any unit vector $v\in\mathbb{R}^{d}$, consider the set
\begin{align*}
    B_{x^{-}, v}&=\left\{\hat{x}\in\mathcal{W}_{loc}^{u}: v\perp \xi(\hat{x})\right\}\\
    &=\left\{\hat{x}\in\mathcal{W}_{loc}^{u}:  \xi(\hat{x})\in v^{\perp}\right\}.\\
    \intertext{Hence, using Theorem \ref{pesoCero}}
    \hat{\mu}_{x^{-}}^{u}(B_{x^{-}, v})&=\hat{\mu}_{x^{-}}^{u}\left\{\hat{x}\in\mathcal{W}_{loc}^{u}:  \xi(\hat{x})\in v^{\perp}\right\}\\
    &=\int \delta_{\xi(\hat{x})}(v^{\perp})d\hat{\mu}_{x^{-}}^{u}(\hat{x})\\
    &=m_{x^{-}}(v^{\perp})\\
    &=0.
\end{align*}
Thus, from (\ref{limiteOseledets}), any sequence $v_{n}\to v$ and all elements $\hat{x}$ in the set of full measure $B_{x^{-}, v}^{c}$ with respect to $\hat{\mu}_{x^{-}}^{u}$ satisfy
$$\lim_{n\to\infty}\dfrac{1}{n}\log{\norm{A^{(n)}(\hat{x})v_{n}}}=\lambda_{1}(A,\hat{\mu}).$$
Since this happens in every local unstable set, we obtain that for $\hat{\mu}$-a.e $\hat{x}$
$$\lim_{n\to\infty}\dfrac{1}{n}\log{\norm{A^{(n)}(\hat{x})v_{n}}}=\lambda_{1}(A,\hat{\mu}).$$
Using the fact that $A$ is a two-step cocycle
$$\dfrac{1}{n}\log\norm{A^{(n)}}\leq \max_{x_{0},x_{1}\in X}\log\norm{A(x_{0},x_{1})},$$
we use the dominated convergence theorem to conclude that
\begin{equation}
  \lim_{n\to\infty}\dfrac{1}{n}\int\log{\norm{A^{(n)}(\hat{x})v_{n}}}d\hat{\mu}(\hat{x})=\lambda_{1}(A,\hat{\mu}).  \label{limiteIntegral}
\end{equation}
We will finish by noting that the limit happens uniformly for all vectors $\norm{v}=1$. Indeed, if the limit was not uniform, fixing some $\epsilon>0$ we could find a sequence of natural numbers $\{n_{i}\}\subset \mathbb{N}$ and a sequence of unit vectors $\{v_i\}\subset \mathbb{S}^{d-1}$ such that
$$\dfrac{1}{n_i}\int\log{\norm{A^{(n_i)}(\hat{x})v_{i}}}d\hat{\mu}(\hat{x})+\epsilon<\lambda_{1}(A,\hat{\mu}).$$
In particular, taking a convergent subsequence $v_{m}\to v$
$$\limsup_{m\to\infty}\dfrac{1}{m}\int\log{\norm{A^{(m)}(\hat{x})v_{m}}}d\hat{\mu}(\hat{x})<\lambda_{1}(A,\hat{\mu}),$$
which is impossible by (\ref{limiteIntegral}).
\end{proof}

Now we will proceed to prove a stronger version of Lemma \ref{PrimerLDE} which requires the use of strong irreducibility and the simplicity of the top Lyapunov exponent since we need to use Lemma \ref{LemmaIntegral}. The proof is similar to that of Lemma \ref{PrimerLDE}, so we will shorten some steps since they were previously explained. 

\begin{lemma} \label{SegundoLDE}
     Let $\mu$ be the equilibrium state of a Hölder potential. Consider a strongly irreducible 2-step cocycle $A$ such that $\lambda_1(\mu)>\lambda_2(\mu)$. Given $\epsilon>0$, there exist $M>0, \beta>0$, such that for all $n\in\mathbb{N}$ and all unit vectors $v\in\mathbb{R}^{d}$
    $$\mu\left\{x\in\Sigma:\left\lvert \frac{1}{n}\log\norm{A^{(n)}(x)v}-\lambda_{1}(A,\hat{\mu})\right\rvert>\epsilon \right\}<Me^{-\beta n}.$$
\end{lemma}
\begin{proof}
    By rescaling the cocycle, we may assume that $\lambda_1(A,\hat{\mu})=0$. 
    Given $\epsilon>0$ and any unit vector $v\in\mathbb{R^{d}}$, fix $N\in \mathbb{N}$ such that
    $$\left|\int\frac{1}{N}\log\norm{A^{(N)}(x)v}d\mu\right|<\frac{\epsilon}{4}.$$
    Note that Lemma \ref{LemmaIntegral} ensures that $N$ does not depend on $v$.\\
    Using Lemma \ref{exponentialmixing}, we can find $k\in \mathbb{N}$ sufficiently large such that for all $i\geq k$
    \begin{align*}
        &\left|\dfrac{1}{N \mu(D_{i}(\omega))}\int_{D_{k}(\omega)}\log\norm{A^{(N)}(T^{i}x)v}d\mu\right|<\dfrac{\epsilon}{2}. 
    \end{align*}
    Consider the random variables $Y_{i}:\Sigma\to\mathbb{R}$ defined by $Y_{i}(x):=\log\norm{A^{(i)}(x)v}$. Denote $\mathcal{B}_{i}$ as the $\sigma$-algebra generated by the set of cylinders of length $(i+1)$. Hence
    \begin{align*}
        \mathbb{E}(Y_{N+i}|\mathcal{B}_{i})(\omega)=\dfrac{1}{\mu(D_{i}(\omega))}\int_{D_{i}(\omega)}\log\norm{A^{(N+i)}(z)v}d\mu(z).
    \end{align*}
    Thus, using \eqref{desigualdad1}, we obtain
    \begin{align*}
        \left|\dfrac{1}{N}\mathbb{E}(Y_{N+i}-Y_{i}|\mathcal{B}_i)\right |&=\left |\dfrac{1}{N \mu(D_{i}(\omega))}\int_{D_{i}(\omega)}\log\dfrac{\norm{A^{(N+i)}(z)v}}{\norm{A^{(i)}(z)v}}d\mu(z)\right|\\
        &\leq\left|\dfrac{1}{N \mu(D_{i}(\omega))}\int_{D_{i}(\omega)}\log \norm{A^{(N)}(T^{i}z)\frac{A^{(i)}(z)v}{\norm{A^{(i)}(z)v}}}d\mu(z)\right|\\
        &\leq \dfrac{\epsilon}{2}. \numberthis \label{desigualdad2}
    \end{align*}
    Consider $t\in\mathbb{N}$ such that $tN\geq k$ and define the sequence of random variables
    $$X_{n}=Y_{nN}-\sum_{l=t}^{n}\mathbb{E}(Y_{lN}-Y_{(l-1)N}|\mathcal{B}_{(l-1)N}).$$
    By the same calculation as before, we find that $\left\{X_{n}\right\}$ is a martingale difference adapted to $\left\{\mathcal{B}_{nN}\right\}$.
    
    Using the partition $D_{nN}(\omega)=\sqcup_{j}D_{(n+1)N}(x_{j})$ for appropriate $x_{j}\in \Sigma$, we obtain
    \begin{align*}
        |X_{n+1}-X_{n}|&
        = \sum_{x_{j}\in D_{nN}(\omega)}\dfrac{\mu(D_{(n+1)N}(x_{j}))}{\mu(D_{nN}(\omega))}\left|\log\norm{A^{((n+1)N)}(\omega)v}-\log\norm{A^{((n+1)N)}(x_{j})v}\right|.
    \end{align*}
    Hence, if we denote $\gamma_{j}=\dfrac{\mu(D_{(n+1)N}(x_{j}))}{\mu(D_{nN}(\omega))}$, we obtain
    \begin{align*}
        |X_{n+1}-X_{n}|&\leq \sum_{x_{j}\in D_{nN}(\omega)}\gamma_{j}\left|\log\dfrac{\norm{A^{((n+1)N)}(\omega)v}}{\norm{A^{((n+1)N)}(x_{j})v}}\right|\\
        &\leq\sum_{x_{j}\in D_{nN}(\omega)}\gamma_{j}\max_{z_{1},z_{2}\in\Sigma}\log\dfrac{\norm{A^{(N)}(z_{1})}}{m(A^{(N)}(z_{2}))}\\
        &\leq a,
    \end{align*}
    where $a>0$ is a constant which depends on $N$. As a consequence, using Azuma's bound \cite[Corollary~7.2.2]{Alon2000}, we obtain that
    \begin{align}
        \mu\left\{\left|\frac{1}{n}X_{n}\right|>\epsilon\right\}<2e^{-\frac{n\epsilon^{2}}{2a^{2}}}. \label{Azuma}
    \end{align}
    Consider $\omega\in \Sigma$ such that
    $$\left|\dfrac{1}{nN}\log\norm{A^{(nN)}(\omega)v}\right|>\epsilon.$$
    Note that this is the same as
    $$\left|\dfrac{1}{nN}Y_{nN}(\omega)\right|>\epsilon.$$
    Hence, using \eqref{desigualdad2}
    \begin{align*}
        \left|\dfrac{1}{nN}X_{n}(\omega)\right|&=\left|\dfrac{1}{nN}\left(Y_{nN}(\omega)-\sum_{l=t}^{n}\mathbb{E}(Y_{lN}-Y_{(l-1)N}|\mathcal{B}_{(l-1)N})(\omega)\right)\right|\\
        &\geq \left|\dfrac{1}{nN}Y_{nN}(\omega)\right|-\left|\dfrac{1}{n}\sum_{l=t}^{n}\left(\dfrac{1}{N}\mathbb{E}(Y_{lN}-Y_{(l-1)N}|\mathcal{B}_{(l-1)N})(\omega)\right)\right|\\
        &>\frac{\epsilon}{2}. 
    \end{align*}
    Thus, we can conclude from \eqref{Azuma}
    \begin{align*}
        \mu\left\{\omega:\left|\dfrac{1}{nN}\log\norm{A^{(nN)}(\omega)v}\right|>\epsilon\right\}&\leq \mu\left\{\omega:\left|\dfrac{1}{n}X_{n}\right|>\dfrac{N\epsilon}{2}\right\}\\
        &\leq M e^{-\frac{N\epsilon^2}{8a^{2}}(nN)}.
    \end{align*}
    Taking $s\geq t$ sufficiently large, we see that for $1\leq l\leq N-1$
        \begin{align*}
        &\mu\left\{\log\norm{A^{(nN+l)}(x)v}>2(nN+l)\epsilon\right\} \\
        &\leq \mu\left\{\log\norm{A^{((n+1)N)}(x)v}>(n+1)N\epsilon\right\}\\
        &\leq Me^{-\frac{N\epsilon^2}{8a^{2}}(nN+l)}.
    \end{align*}
    We conclude by sufficiently enlarging the constant $M=M(N)$ so that the statement is true for $0\leq m\leq sN-1$. 
\end{proof}

We observe that the previous result is similar to \cite[Theorem~7.1]{Duarte}, which is proved in the context of pinching and twisting cocycles. Their result does not require the cocycle to be locally constant and yields estimates for nearby fiber bunched cocycles. However, unlike Lemma \ref{SegundoLDE}, their result does not contain information about the growth of a fixed unit vector under the cocycle. Such information will be necessary for the proofs in Section \ref{Section5}.

We will finish this section with a LDE for the growth of the norm of the cocycle $A$. The following argument is standard. 

\begin{corollary}\label{Corolario}
    Let $\mu$ be the equilibrium state of a Hölder potential. Consider a strongly irreducible 2-step cocycle $A$ such that $\lambda_1(\mu)>\lambda_2(\mu)$. Given $\epsilon>0$, there exist $M>0, \beta>0$, such that for all $n\in\mathbb{N}$
    $$\mu\left\{x\in\Sigma:\left\lvert \frac{1}{n}\log\norm{A^{(n)}(x)}-\lambda_{1}(A,\hat{\mu})\right\rvert>\epsilon \right\}<Me^{-\beta n}.$$ 
\end{corollary}
\begin{proof}
    As before, we assume $\lambda_{1}=0$. Consider $v_{1},\dots, v_{d}\in\mathbb{R}^{d}$ an orthonormal basis of and for $G\in\text{GL}(d,\mathbb{R})$, define the norm
    $$\norm{G}_{\mathcal{V}}=\max_{1\leq j\leq d}\norm{Gv_{j}}.$$
    We see that 
    \begin{align*}
        \mu\left\{x\in\Sigma: \dfrac{1}{n}\log\norm{A^{(n)}(x)}_{\mathcal{V}}>\epsilon\right\}\leq\sum_{1\leq j\leq d}\mu\left\{x\in\Sigma: \dfrac{1}{n}\log\norm{A^{(n)}(x)v_j}>\epsilon\right\}.
    \end{align*}
    Thus, applying Lemma \ref{SegundoLDE} to the right hand side and since all norms are equivalent in finite dimensional vector spaces, the result follows. 
\end{proof}

\section{Distance between direction of maximum growth and hyperplane of least growth} \label{Section5}
In this section, we will closely follow the scheme used in Aoun and Sert \cite[Section~3]{aoun}.

Throughout this section, we will assume that $A:\hat{\Sigma}\to \text{GL}(d,\mathbb{R})$ is a 2-step strongly irreducible cocycle whose first Lyapunov exponent is simple. Moreover, we will denote $\lambda_{1}=\lambda_{1}(A,\hat{\mu})=\lambda_{1}(A_{*},\hat{\mu})$ and $\lambda_{2}=\lambda_{2}(A,\hat{\mu})=\lambda_{2}(A_{*},\hat{\mu})$. Additionally, we fix $0<\epsilon<(\lambda_1-\lambda_2)/20$.

We will begin by showing that except for a set of exponentially small measure, iterations of the cocycle applied to any unit vector exponentially approach the direction of maximum growth.
\begin{lemma} \label{Lemma5.1}
    There exist positive constants $\alpha>0, \beta>0, M>0$ such that for every unit vector $v\in \mathbb{R}^d$ 
    $$\hat\mu\{d(u(A^{(n)}(\hat{x})), \overline{A^{(n)}(\hat{x})v})\geq e^{-\alpha n}\}\leq Me^{-\beta n}.$$
\end{lemma}
\begin{proof}
    Take $v\in\mathbb{R}^d$ such that $\norm{v}=1$. From Aoun-Sert \cite{aoun}, we know that for any $g\in \text{GL}(d,\mathbb{R})$:
    \begin{equation} \label{eq:1}
        d(u(g),gv)\leq \dfrac{a_2(g)}{a_1(g)}\dfrac{\Vert g\Vert}{\Vert gv\Vert}=\dfrac{\Vert\mathsf{\Lambda}^2 g\Vert}{\Vert g\Vert^{2}}\dfrac{\Vert g\Vert}{\Vert gv\Vert}.
    \end{equation}
    Note that $\mathsf{\Lambda}^{2} A:{\hat{\Sigma}}\to \text{Gl}(\frac{d(d-1)}{2},\mathbb{R})$ is a 2-step linear cocycle whose first Lyapunov exponent is $\lambda_{1}+\lambda_{2}$. Thus, applying Lemma \ref{PrimerLDE} to  $\mathsf{\Lambda}^{2} A$ and Lemma \ref{SegundoLDE} and Corollary \ref{Corolario} to $A$, we obtain that there exist $M>0$, $\beta>0$ and a set of measure larger than $1-Me^{-\beta n}$ such that for $\alpha=(\lambda_1-\lambda_2)/2$
    \begin{align*}
        d\left(u(A^{(n)}(\hat{x})), \overline{A^{(n)}(\hat{x})v}\right)\leq \dfrac{\left\Vert\mathsf{\Lambda}^2 A^{(n)}(\hat{x})\right\Vert}{\left\Vert A^{(n)}(\hat{x})\right\Vert^2}\dfrac{\left\Vert A^{(n)}(\hat{x})\right\Vert}{\left\Vert A^{(n)}(\hat{x})v\right\Vert}\leq e^{n(\lambda_2-\lambda_1+5\epsilon)}\leq e^{-n\alpha}.
    \end{align*}
\end{proof}
\vspace{-0.1cm}
We will use the following notation for the rest of the paper 
$$A^{[n]}(\hat{x})=(A^{(n)}(\hat{x}))^{*}=A_{*}^{(n)}(T^{n}\hat{x}).$$
Note that using Lemma \ref{SegundoLDE} with $A_{*}$ and using the invariance of the measure with respect to $T$, we obtain the existence of $M>0, \beta>0$ such that for all unit vectors $v\in\mathbb{R}^{d}$

\begin{align*}
&\hat{\mu}\left\{\hat{x}\in\hat{\Sigma}:\left\lvert \frac{1}{n}\log\norm{A^{[n]}(\hat{x})v}-\lambda_{1}(A,\hat{\mu})\right\rvert>\epsilon \right\}\\
&=\hat{\mu}\left\{\hat{x}\in\hat{\Sigma}:\left\lvert \frac{1}{n}\log\norm{A^{(n)}_{*}(T^{n}\hat{x})v}-\lambda_{1}(A,\hat{\mu})\right\rvert>\epsilon \right\}\\
&=\hat{\mu}\left\{\hat{x}\in\hat{\Sigma}:\left\lvert \frac{1}{n}\log\norm{A_{*}^{(n)}(\hat{x})v}-\lambda_{1}(A,\hat{\mu})\right\rvert>\epsilon \right\}\\
&<Me^{-\beta n}.    
\end{align*}
Similarly, using Lemma \ref{PrimerLDE}, we obtain that
$$\hat{\mu}\left\{\hat{x}\in\hat{\Sigma}: \frac{1}{n}\log\left(\mathsf{\Lambda}^{2}\norm{A^{[n]}(\hat{x})}\right)-(\lambda_{1}+\lambda_{2})>\epsilon \right\}<Me^{-\beta  n}.$$
Then, we can prove the following lemma, which shows that the direction of maximum growth of iterations of the cocycle with right increments stabilizes at an exponential rate with the possible exception of a set of exponentially small measure. 
\begin{lemma} \label{Lemma5.2}
    There exist positive constants $\alpha>\epsilon, \beta, M$ and $N\in \mathbb{N}$ such that for all $n\geq N$
    $$\hat{\mu}\left(d(u(A^{[4n]}(\hat{x})), u(A^{[n]}(\hat{x})))\geq e^{-\alpha n}\right)\leq Me^{-\beta n}$$
\end{lemma}

This lemma can be deduced from \cite[Proposition~4.18]{klein} using the fact that the gap between the first and the second Lyapunov exponents generate a gap in the first two singular values of $A^{(n)}(\hat{x})$ for $n$ large enough (see Avalanche Principle \cite[Section~3.3]{klein}). For the convenience of the reader, we present a self-contained proof based on our previous results. 

\begin{proof}
    Take $v\in\mathbb{R}^d$ a unit vector. Then, for all $C>0$ and $n\in\mathbb{N}$:
    \begin{align*}
        &\hat{\mu}\left(d(u(A^{[4n]}(\hat{x})), u(A^{[n]}(\hat{x})))\geq e^{-Cn}\right)\\
        &\leq\hat{\mu}\left(d(u(A^{[4n]}(\hat{x})),\overline{A^{[4n]}(\hat{x})v})\geq \dfrac{e^{-Cn}}{3}\right)+
        \hat{\mu}\left(d(\overline{A^{[n]}(\hat{x})v}, u(A^{[n]}(\hat{x})))\geq \dfrac{e^{-Cn}}{3}\right)
        \\&\hspace{4cm}+\hat{\mu}\left(d(\overline{A^{[4n]}(\hat{x})v},\overline{A^{[n]}(\hat{x})v})\geq \dfrac{e^{-Cn}}{3}\right)
    \end{align*}
    From Lemma \ref{Lemma5.1}, we know there exist $\alpha_1>0 ,M_1>0, \beta>0$ such that
    \begin{align*}
        \hat{\mu}\left(d(u(A^{[4n]}(\hat{x})),\overline{A^{[4n]}(\hat{x})v})\geq e^{-4\alpha_1 n}\right)&\leq M_1e^{-4n\beta_1}, \\
        \hat{\mu}\left(d(\overline{A^{[n]}(\hat{x})v}, u(A^{[n]}(\hat{x})))\geq e^{-\alpha_1 n}\right)&\leq Me^{-n\beta_1}.
    \end{align*}
    Hence, we need only to find bounds for the measure of $\left\{d(\overline{A^{[4n]}(\hat{x})v}, \overline{A^{[n]}(\hat{x})v}\right)\geq e^{-\alpha_2n}\}$ for some constant $\alpha_2>0$.
    
    We have the following bound
    \begin{align*}
    d(\overline{A^{[4n]}(\hat{x})v}, \overline{A^{[n]}(\hat{x})v})&=\dfrac{\Vert \mathsf{\Lambda}^2 A^{[n]}(\hat{x})(A^{[3n]}(T^n \hat{x})v\wedge v) \Vert}{\Vert A^{[4n]}(\hat{x})v\Vert \Vert A^{[n]}(\hat{x})v\Vert}\\    
     &\leq \dfrac{\Vert \mathsf{\Lambda}^2 A^{[n]}(\hat{x})\Vert \Vert A^{[3n]}(T^n \hat{x})\Vert }{\Vert A^{[4n]}(\hat{x})v\Vert \Vert A^{[n]}(\hat{x})v\Vert}.
    \end{align*}
    Hence, using Lemma \ref{PrimerLDE} for $\mathsf{\Lambda}^2 A^{[n]}(\hat{x})$, Lemma \ref{SegundoLDE} on $ A^{[4n]}(\hat{x})v$ and $A^{[n]}(\hat{x})v$ and Corollary \ref{Corolario} on $A^{[3n]}(T^n \hat{x})$, we obtain that there exist $M_2>0$ and $\beta_2>0$ and a set of measure greater than $1-M_2e^{-\beta_2 n}$ such that
    \begin{align*}
        \dfrac{\Vert \mathsf{\Lambda}^2 A^{[n]}(\hat{x})\Vert \Vert A^{[3n]}(T^n \hat{x})\Vert}{\Vert A^{[4n]}(\hat{x})v\Vert \Vert A^{[n]}(\hat{x})v\Vert}\leq e^{n(\lambda_2-\lambda_1+9\epsilon)}<e^{\alpha_2 n},
    \end{align*}
    where $\alpha_2=\frac{\lambda_1-\lambda_2}{10}$. We conclude that that:
    \begin{align*}
        \hat\mu\left\{d(A^{[4n]}(\hat{x})v, A^{[n]}(\hat{x})v)> e^{-\alpha_2n} \right\}&\leq \hat{\mu}\left\{\hat{x}\in\hat{\Sigma}:\dfrac{\Vert \mathsf{\Lambda}^2 A^{[n]}(\hat{x})\Vert \Vert A^{[3n]}(T^n \hat{x})\Vert}{\Vert A^{[4n]}(\hat{x})v\Vert \Vert A^{[n]}(\hat{x})v\Vert}> e^{-\alpha_2 n}\right\}\\
        &\leq M_2e^{-n\beta_{2} }. 
    \end{align*} 
\end{proof}
\begin{remark} \label{Remark5.2}
    Note that we had to use a cocycle with right increments to properly factorize the matrices applied to both elements of the wedge product. Thus, since $A^{(4n)}(\hat{x})$ and $A^{(n)}(T^{3n}\hat{x})v$ coincide in the first $n$ left factors, we are able to obtain the same bounds for $d(\overline{A^{(4n)}(\hat{x})v}, \overline{A^{(n)}(T^{3n}\hat{x})v})$. Hence, replicating the process of Lemma \ref{Lemma5.2}, we get
    $$\hat{\mu}\left(d(u(A^{(4n)}(\hat{x})), u(A^{(n)}(T^{3n}\hat{x})))\geq e^{-\alpha n}\right)\leq Me^{-\beta n}.$$
\end{remark}
The following lemma shows that the direction of maximum growth of iterations of the cocycle does not converge exponentially to any hyperplane, excepting maybe on a set of exponentially small measure. 
\begin{lemma} \label{Lemma5.3}
    There exist positive constants $\beta>0$ and $M>0$ such that for every $n\geq N$ where $4\leq e^{N\epsilon}$ and for every unit vector $w\in \mathbb{R}^d$ 
    $$\hat\mu\{d(u(A^{(n)}(\hat{x})), w^{\perp})\leq 3e^{-3\epsilon n}\}\leq Me^{-\beta n}.$$
\end{lemma}
\begin{proof}
    From Bourgain et all \cite[Lemma~4.1]{bourgain} we have the following inequality: For all $g\in GL(d, \mathbb{R})$ such that $a_{2}(g)/a_1(g)<1$ and all unit vectors $w\in \mathbb{R}^d$
    $$\dfrac{\Vert g^*w\Vert}{\Vert g^{*}\Vert}\leq d(u(g),w^{\perp})+\dfrac{a_2(g)}{a_2(g)}.$$
    Using Lemma \ref{PrimerLDE} on $\mathsf{\Lambda}^2 A^{(n)}$, Lemma \ref{SegundoLDE} on $A^{[n]}(\hat{x})w$ and Corollary \ref{Corolario} on $A^{(n)}(\hat{x})$ and $ A^{[n]}(\hat{x})$,  we obtain $M>0$ and $\beta>0$ such that
    \begin{align*}
        \hat{\mu}\{d(u(A^{(n)}(\hat{x})), w^{\perp}) \leq e^{-2n\epsilon}&-e^{n(\lambda_2-\lambda_1+3\epsilon)} \}\\
        &\leq \hat{\mu}\left\{\dfrac{\left\Vert A^{[n]}(\hat{x})w\right\Vert}{\left\Vert A^{[n]}(\hat{x})\right\Vert}-\dfrac{\left\Vert\mathsf{\Lambda}^2 A^{(n)}(\hat{x})\right\Vert}{\left\Vert A^{(n)}(\hat{x})\right\Vert^2}\leq e^{-2n\epsilon}-e^{n(\lambda_2-\lambda_1+3\epsilon)}\right\} \\
        &\leq Me^{-\beta n}.
    \end{align*}
    Finally, since $4\leq e^{n\epsilon}$
    $$3e^{-3\epsilon n}\leq e^{-2\epsilon n}-e^{n(\lambda_2-\lambda_1+3\epsilon)},$$
    we conclude
    $$\hat\mu\{d(u(A^{(n)}(\hat{x})), w^{\perp})\leq 3e^{-3\epsilon n}\}\leq Me^{-\beta n}.$$
\end{proof}
We need to show that the direction of maximum growth and the hyperplane of least growth of iterates of the cocycle $A$ only get exponentially close on a exponentially small set. In order to get the expected result, we will take advantage of the exponential mixing of the equilibrium measure and the local constancy of the cocycle, which will allow us to compare the distance of the hyperplane of least growth in \say{the present} with the direction of maximum growth in \say{the future} as almost independent objects. We will need the following tools:
\begin{align*}
P_n=\{(\overline{v},\overline{w})\in \mathbb{RP}^{d-1}\times \mathbb{RP}^{d-1} &:  d(\overline{v}^{\perp}, \overline{w})\leq 3e^{-3\epsilon n}\},\\
C_n(\overline{w})=\{[x_0\dots x_{n-1}x_{n}]\subseteq \hat{\Sigma} &: u(A^{(n)}(\hat{x}))=\overline{w}\},\\
C_n^{*}(\overline{v})=\{[x_0\dots x_{n-1}x_{n}]\subseteq \hat{\Sigma} &: s(A^{(n)}(\hat{x}))=\overline{v}^{\perp}\}.
\end{align*}
Consider the following maps:
\begin{align*}
    \phi^1_n: \hat{\Sigma}\longrightarrow &\hspace{0.1cm}\mathbb{RP}^{d-1}\\
    \hat{x}\longmapsto & \hspace{0.1cm} s(A^{(n)}(\hat{x}))^{\perp},\\
    \phi^{2}_{n}: \hat{\Sigma}\longrightarrow & \hspace{0.1cm} \mathbb{RP}^{d-1}\\
    \hat{x}\longmapsto  & \hspace{0.1cm} u(A^{(n)}(T^{3n}\hat{x})),\\
    \Phi_{n}=& \hspace{0.1cm}(\phi^{1}_n,\phi^{2}_n).
    \intertext{From them, we can define the following pushforward measures:}
    \eta_{n}=& \hspace{0.1cm}\Phi_{n*} \hat{\mu},\\
    \nu^1_{n}=& \hspace{0.1cm}\phi^{1}_{n*} \hat{\mu},\\
    \nu^2_{n}=& \hspace{0.1cm}\phi^{2}_{n*} \hat{\mu}.\\
\end{align*}
\begin{lemma}[\textbf{Almost independence}] \label{Lemma5.4}
    There exist $K>0$ and $\gamma\in (0,1)$ such that
    $$|\eta_{n}(P_n)-\nu^1_{n}\times\nu^2_{n}(P_n)|\leq K\gamma^{2n}.$$
\end{lemma}
\begin{proof}
    Let us calculate the measure of $P_n$ under each measure:
    \begin{align*}
        \eta_{n}(P_n)=&\sum_{(\overline{v_1},\overline{v_2})\in P_n}\hat{\mu}\left\{\hat{x}\in\hat{\Sigma}:\phi^1_{n}(\hat{x})=\overline{v_1}^{\perp}, \phi^2_{n}(\hat{x})=\overline{v_2} \right\}\\
        =&\sum_{(\overline{v_1},\overline{v_2})\in P_n}\hat{\mu}\left(C_{n}^{*}(\overline{v_1})\cap T^{-3n}C_{n}(\overline{v_2})\right).\\
        \intertext{Similarly}
        \nu^1_{n}\times\nu^2_{n}(P_n)=&\sum_{(\overline{v_1},\overline{v_2})\in P_n}\hat{\mu}(C_{n}^{*}(\overline{v_1}))\hat{\mu}(T^{-3n}C_{n}(\overline{v_2}))\\
        \nu^1_{n}\times\nu^2_{n}(P_n)=&\sum_{(\overline{v_1},\overline{v_2})\in P_n}\hat{\mu}(C_{n}^{*}(\overline{v_1}))\hat{\mu}(C_{n}(\overline{v_2}))
    \end{align*}
   Due to the exponential mixing on unions of cylinders for equilibrium states due to Bowen \cite[Theorem~1.25]{Bowen2008}, we know that there is $\gamma\in(0,1)$ and $K>0$ such that for any pair of unions of cylinders $O_1,O_2\subseteq \hat{\Sigma}$ of length $(n+1)$ 
   $$\left|\hat{\mu}(O_1\cap T^{-3n}O_2)-\hat{\mu}(O_1)\hat{\mu}(O_2)\right|\leq K\hat{\mu}(O_1)\hat{\mu}(O_2)\gamma^{3n-n}.$$
   Hence
   \begin{align*}
       \left |\eta_{n}(P_n)-\nu^1_{n}\times\nu^2_{n}(P_n)\right|&\leq \sum_{(\overline{v_1},\overline{v_2})\in P_n}\left|\hat{\mu}\left(C_{n}^{*}(\overline{v_1})\cap T^{-3n}C_{n}(\overline{v_2})\right)-\hat{\mu}(C_{n}^{*}(\overline{v_1}))\hat{\mu}(C_{n}(\overline{v_2}))\right|\\
       &\leq \sum_{(\overline{v_1},\overline{v_2})\in P_n}K\hat{\mu}(C_{n}^{*}(\overline{v_1}))\hat{\mu}(C_{n}(\overline{v_2}))\gamma^{2n}\\
       &\leq  K\gamma^{2n}.
   \end{align*}
\end{proof}
The following is the main lemma of this section. We will take advantage of the fact that the hyperplanes of least growth stabilize exponentially fast, that the distance between directions of maximum growth in the present and in the future do not stray off and that the direction of maximum growth in the future and the hyperplane of least growth in the present are almost independent objects. 
\begin{lemma} \label{Lemma5.5}
    Take $N\in \mathbb{N}$ such that $\frac{\lambda_1-\lambda_2}{10}>\frac{\log{3}}{N}$ and $4\leq e^{N\epsilon}$. Then for every $n\geq 4N$, there exist $M>0, \beta>0, K>0$ and $\gamma\in (0,1)$ such that
    $$\hat{\mu}\left(d(u(A^{(n)}(\hat{x})), s(A^{(n)}(\hat{x})))\leq e^{-3\epsilon \lfloor n/4\rfloor}\right)\leq Me^{-\beta \lfloor n/4\rfloor}+K\gamma^{n/2}.$$
\end{lemma}
\begin{proof}
    By Lemma \ref{Lemma5.2} and Remark \ref{Remark5.2}, there exist $\alpha_1>\epsilon, \beta_{1}>0$ and $M_{1}>0$ such that
    \begin{align*}
    \hat{\mu}\left(d(u(A^{[n]}(\hat{x})), u(A^{[\lfloor n/4 \rfloor]}(\hat{x})))\geq e^{-\alpha_1 \lfloor n/4 \rfloor}\right)&\leq M_1 e^{-\beta_{1} \lfloor n/4 \rfloor},   \numberthis \label{PasoUaS}\\
    \hat{\mu}\left(d(u(A^{(n))}(\hat{x})), u(A^{(\lfloor n/4\rfloor )}(T^{\lceil 3n/4\rceil}\hat{x})))\geq e^{-\alpha_{1}\lfloor n/4 \rfloor}\right)&\leq M_1e^{-\beta_1 \lfloor n/4 \rfloor}.
    \end{align*}
    Since $s(G)=\{\overline{w}\in\mathbb{RP}^{d-1}: \langle w, u(G^*)\rangle=0\}$, from (\ref{PasoUaS}) we have that
    $$\hat{\mu}\left(d(s(A^{(n)}(\hat{x})), s(A^{(\lfloor n/4 \rfloor)}(\hat{x})))\geq e^{-\alpha_1 \lfloor n/4 \rfloor}\right)\leq M_1e^{-\beta_{1} \lfloor n/4 \rfloor}.$$
    Denote 
    \begin{multline*}
    \Omega_{n}=\left\{d(s(A^{(n)}(\hat{x})), s(A^{(\lfloor n/4 \rfloor)}(\hat{x})))\geq e^{-\alpha_1 \lfloor n/4 \rfloor}\right\} \cup \left\{d(u(A^{(n))}(\hat{x})), u(A^{(\lfloor n/4\rfloor )}(T^{\lceil 3n/4\rceil}\hat{x})))\geq e^{-\alpha_1 \lfloor n/4 \rfloor}\right\}.    
    \end{multline*}
    Consider the following decomposition of events 
    \begin{multline*}
        \{d(u(A^{(n)}(\hat{x})), s(A^{(n)}(\hat{x})))\leq e^{-3\epsilon \lfloor n/4 \rfloor}\}
        =\{d(u(A^{(n)}(\hat{x})), s(A^{(n)}(\hat{x})))\leq e^{-3\epsilon \lfloor n/4 \rfloor}\}\cap(\Omega_{n}\cup \Omega_{n}^{C}).
    \end{multline*}
    Therefore, we have the following bounds for the measure of the event
    \begin{align*} 
        \hat{\mu}(d(u(A^{(n)}(\hat{x}))&, s(A^{(n)}(\hat{x})))\leq e^{-3\epsilon \lfloor n/4\rfloor})\\
        &\leq 2M_1 e^{-\beta_1 \lfloor n/4 \rfloor}+\hat{\mu}\left(d(u(A^{(\lfloor n/4\rfloor )}(T^{\lceil 3n/4\rceil}\hat{x})), s(A^{(\lfloor n/4 \rfloor)}(\hat{x})))\leq 3e^{-3\epsilon \lfloor n/4 \rfloor}\right)\\
        &\leq 2M_{1}e^{-\beta_1 \lfloor n/4 \rfloor}+\eta_{\lfloor n/4 \rfloor}(P_{\lfloor n/4 \rfloor}),\\
        \intertext{using Lemma \ref{Lemma5.4} }
        &\leq 2M_1 e^{-\beta_1 \lfloor n/4 \rfloor}+\nu^1_{\lfloor n/4 \rfloor}\times \nu^2_{\lfloor n/4 \rfloor} (P_{\lfloor n/4 \rfloor})+K\gamma^{n/2}. \numberthis \label{eq:2}
    \end{align*}
    Thus, using Lemma \ref{Lemma5.3}, we can find $M_2>0$ and $\beta_2>0$ such that
    \begin{align*}\label{eq:3}
        &\nu^1_{\lfloor n/4 \rfloor}\times \nu^2_{\lfloor n/4 \rfloor} (P_{\lfloor n/4 \rfloor})\\
        &=\sum_{\overline{w}\in \phi^1_{\lfloor n/4 \rfloor} (\hat{\Sigma})}\hat{\mu}\left(u(A^{[\lfloor n/4 \rfloor]}(\hat{x}))=\overline{w}\right)\hat{\mu}\left(d(u(A^{(\lfloor n/4\rfloor )}(T^{\lceil 3n/4\rceil}\hat{x})), w^{\perp})\leq 3e^{-3\epsilon \lfloor n/4 \rfloor}\right)\\
        &\leq \sup_{\Vert v\Vert=1}\hat{\mu}\left(d(u(A^{(\lfloor n/4\rfloor )}(T^{\lceil 3n/4\rceil}\hat{x})), v^{\perp})\leq 3e^{-3\epsilon \lfloor n/4 \rfloor}\right)\sum_{\overline{w}\in \phi^1_{\lfloor n/4 \rfloor} (\hat{\Sigma})}\hat{\mu}\left(u(A^{[\lfloor n/4 \rfloor]}(\hat{x}))=\overline{w}\right)\\
        &\leq M_2e^{-\beta_2 \lfloor n/4 \rfloor}. \numberthis
    \end{align*}
    From (\ref{eq:2}) and (\ref{eq:3}), taking $M=\max\{M_1,M_2\}$ and $\beta=\min\{\beta_{1},\beta_{1}\}$ we conclude that
    $$\hat{\mu}\left(d(u(A^{(n)}(\hat{x})), s(A^{(n)}(\hat{x})))\leq e^{-3\epsilon \lfloor n/4\rfloor}\right)\leq 3Me^{-\beta \lfloor n/4 \rfloor}+K\gamma^{n/2}.$$
\end{proof}

\section{Proof of Main Theorem} \label{Section6}

From Benoist and Quint, we have the following lemma which will help us conclude the theorem.
\begin{lemma}[\protect{\cite[Lemma~14.14]{RandomWalks}}]\label{LemmaBenoistQuint}
    Let $g\in GL(d, \mathbb{R})$. If $d(u(g), s(g))>2\sqrt{\dfrac{a_{2}(g)}{a_{1}(g)}}$, then
    $$\dfrac{d(u(g), s(g))}{2}\leq \frac{\rho(g)}{\norm{g}}.$$
\end{lemma}
Now, we proceed to prove the main theorem of the paper. 
\begin{proof}[Proof of Theorem \ref{TeoremaPrincipal}]
    First, let us assume that $\lambda_1(\hat{\mu},A)>\lambda_2(\hat{\mu},A)$. By Lemma \ref{PrimerLDE} applied to $\mathsf{\Lambda}^{2}A$, Corollary \ref{Corolario} applied to $A$ and by Lemma \ref{Lemma5.5}, for sufficiently large $n \in\mathbb{N}$ there exist $M_1>0$ $\beta_1>0$ and $0<\gamma<1$ such that  
    \begin{align*}
        \hat{\mu}\left(\dfrac{\left\Vert\mathsf{\Lambda}^2 A^{(n)}(\hat{x})\right\Vert}{\left\Vert A^{(n)}(\hat{x})\right\Vert^2}< e^{n(\lambda_2-\lambda_1+3\epsilon)}\right)&>1-M_1e^{-\beta_1 n},\\
        \hat{\mu}\left(d(u(A^{(n)}(\hat{x})), s(A^{(n)}(\hat{x})))> e^{-3\epsilon \lfloor n/4\rfloor}\right)&> 1-M_1e^{-\beta_1 \lfloor n/4\rfloor}-K\gamma^{n/2}.
    \end{align*}
    Then, for $n>\frac{3\log{2}}{\lambda_{2}-\lambda_{1}}$
    \begin{align*}
        \hat{\mu}\left(2\sqrt{\dfrac{\left\Vert\mathsf{\Lambda}^2 A^{(n)}(\hat{x})\right\Vert}{\left\Vert A^{(n)}(\hat{x})\right\Vert^2}} < d(u(A^{(n)}(\hat{x})), s(A^{(n)}(\hat{x})))\right)&>1-2M_1e^{-\beta_1 \lfloor n/4\rfloor}-K\gamma^{n/2},
        \intertext{which implies, by Lemma \ref{LemmaBenoistQuint},}
        \hat{\mu}\left(\dfrac{\rho(A^{(n)}(\hat{x}))}{\norm{A^{(n)}(\hat{x})}}> \dfrac{1}{2}e^{-3\epsilon \lfloor n/4\rfloor}\right)&> 1-2M_1e^{-\beta_1 \lfloor n/4\rfloor}-K\gamma^{n/2}.
        \intertext{Therefore, we have that }
        \sum_{n\in \mathbb{N}}\hat{\mu}\left(\dfrac{\rho(A^{(n)}(\hat{x}))}{\norm{A^{(n)}(\hat{x})}}\leq  \dfrac{1}{2}e^{-3\epsilon \lfloor n/4\rfloor}\right)&<\infty.
    \end{align*}
    By Borel-Cantelli's lemma, for every $0\leq \epsilon \leq \frac{\lambda_{1}-\lambda_{2}}{20}$ and for $\hat{\mu}$-almost every $\hat{x}\in \hat{\Sigma}$, there exist $n_{0}(\hat{x},\epsilon)$ such that for $n>n_{0}$:
    $$-\dfrac{3}{4}\epsilon\leq \dfrac{1}{n}\log{\dfrac{\rho(A^{(n)}(\hat{x}))}{\norm{A^{(n)}(\hat{x})}}}\leq 0.$$
    Thus:
    $$\dfrac{1}{n}\log{\dfrac{\rho(A^{(n)}(\hat{x}))}{\norm{A^{(n)}(\hat{x})}}}\xrightarrow[n\to \infty]{a.s.} 0.$$
    The result is obtained by Oseledets' theorem.
    \end{proof}

    Now, we can reduce the proof of Theorem \ref{TeoremaSecundario} to applying Theorem \ref{TeoremaPrincipal} to an exterior product of the cocycle $A$
    \begin{proof}[Proof of Theorem \ref{TeoremaSecundario}]
    If $\lambda_{1}(\hat{\mu},A)=\dots=\lambda_{d}(\hat{\mu},A)$, then the result follows from the fact that $a_{1}(g)\geq \rho(g)\geq a_{d}(g)$ for all $g\in \text{GL}(d,\mathbb{R})$.

    On the other hand, in case that there exists $1\leq t\leq d$ such that $\lambda_{1}(\hat{\mu},A)=\dots=\lambda_{t}(\hat{\mu},A)>\lambda_{t+1}(\hat{\mu},A)$, consider the cocycle given by $B=\bigwedge^{t}A$. Hence, $\lambda_{1}(\hat{\mu},B)>\lambda_{2}(\hat{\mu},B)$, so from Theorem \ref{TeoremaPrincipal} 
    $$\dfrac{1}{n}\log{\rho(B^{(n)}(\hat{x}))}\xrightarrow[n\to \infty]{a.s.}\lambda_{1}(\hat{\mu},B)=t\lambda_{1}(\hat{\mu},A).$$
    Then, since $\rho\left(\bigwedge^{t}g \right)\leq \rho(g)^{t}$ for every $g\in \text{GL}(d,\mathbb{R})$  we obtain on a set of full measure
    $$\lambda_{1}(\hat{\mu}, A)\leq \liminf_{n\to\infty}\dfrac{1}{n}\log{\rho(A^{(n)}(\hat{x}))}.$$
    From Morris \cite[Theorem~1.6]{Morris}, we know that
    $$\lambda_{1}(\hat{\mu}, A)= \limsup_{n\to\infty}\dfrac{1}{n}\log{\rho(A^{(n)}(\hat{x}))}.$$
    We conclude by comparing the last two expressions. 
    \end{proof}
    \section*{Acknowledgments}
    I would like to thank my advisor Jairo Bochi for his constant support, insightful advice and motivation to develop this work. I would also like to extend my thanks to Cagri Sert for many helpful discussions which continuously guided this project.  Finally, I would like to thank the referee for their valuable suggestions to improve the connections of this paper to the available literature and in the overall form of the document.

\end{document}